\documentclass{amsart}
\usepackage{listings}
\usepackage{mathtools}
\usepackage{amssymb}
\usepackage{amsthm}
\usepackage{url}
\usepackage{eucal}
\usepackage{tikz}
\usepackage{cite}
\usetikzlibrary{decorations.markings}
\numberwithin{equation}{section}
\theoremstyle{plain}
\newtheorem{theorem}{Theorem}[section]
\newtheorem{lemma}[theorem]{Lemma}

\newtheorem{corollary}[theorem]{Corollary}
\theoremstyle{definition}
\newtheorem{definition}[theorem]{Definition}
\newtheorem{remark}[theorem]{Remark}
\newtheorem{example}[theorem]{Example}
\newcommand*{\id}{{\mathrm{id}}}

\newcommand*{\R}{{\mathbb R}}                                                 

\newcommand*{\CC}{{\mathbb{C}}}                                               
                                                
\newcommand*{\Hb}{{\mathbb H}}                                                
                                                
\newcommand*{\Vb}{{\mathbb V}}
\newcommand*{\Wb}{{\mathbb W}}

\newcommand*{\rd}{{\mathrm d}}                                                
                                               
\newcommand*{\Gr}{{\mathrm{Gr}}} 
\newcommand{\pa}{\partial}
\newcommand{\tP}{\tilde{P}}
\newcommand{\tG}{\tilde{G}}
\newcommand{\tg}{\tilde{g}}
\newcommand{\tQ}{\tilde{Q}}

\allowdisplaybreaks

\begin{document}
\title[Non-commutative quintic nonlinear Schr\"odinger equation]{Integrability of local and nonlocal non-commutative fourth order quintic nonlinear Schr\"odinger equations}
\author{Simon J.A. Malham}
\date{13th July 2021}
\address{Maxwell Institute for Mathematical Sciences,        
and School of Mathematical and Computer Sciences,   
Heriot-Watt University, Edinburgh EH14 4AS}
\email{S.J.A.Malham@hw.ac.uk}

\begin{abstract}
We prove integrability of a generalised non-commutative fourth order quintic nonlinear Schr\"odinger equation.
The proof is relatively succinct and rooted in the linearisation method pioneered by Ch. P\"oppe.
It is based on solving the corresponding linearised partial differential system to generate
an evolutionary Hankel operator for the `scattering data'. The time-evolutionary solution
to the non-commutative nonlinear partial differential system is then generated by solving 
a linear Fredholm equation which corresponds to the Marchenko equation.
The integrability of reverse space-time and reverse time nonlocal versions,
in the sense of Ablowitz and Musslimani \cite{AM},
of the fourth order quintic nonlinear Schr\"odinger equation are proved contiguously by the approach adopted.
Further, we implement a numerical integration scheme based on the analytical approach
above which involves solving the linearised partial differential system followed by
numerically solving the linear Fredholm equation to generate the solution at any given time. 
\end{abstract}

\maketitle

\section{Introduction}
We prove that a generalised non-commutative fourth order quintic nonlinear Schr\"odinger equation is integrable. 
Here `integrable' means the equation can be linearised.
Precisely though briefly, given time-evolutionary solutions 
to the corresponding linearised equation,
we can generate corresponding solutions to the original nonlinear equation
at any given time by solving a linear integral Fredholm equation at that time.   
The Fredholm equation in question corresponds to the Marchenko equation.
This approach to finding solutions to classical integrable systems such as
the sine--Gordon, Korteweg de Vries, modified Korteweg de Vries, the whole
associated Korteweg de Vries hierarchy and also the nonlinear Schr\"odinger equation
was pioneered by Ch. P\"oppe in a sequence of papers, see P\"oppe~\cite{P83,P84,P-KP},
P\"oppe and Sattinger~\cite{PS88} and Bauhardt and P\"oppe~\cite{BP-ZS}.
Recently Doikou \textit{et al.\/} \cite{DMSW20,DMS20} extended P\"oppe's approach.
First they demonstrated for the Korteweg de Vries and nonlinear Schr\"odinger equation,
only P\"oppe's celebrated kernel product rule is required for the approach to work,
see Doikou \textit{et al.\/} \cite{DMSW20}. Second they demonstrated the approach,
as considered by Bauhardt and P\"oppe~\cite{BP-ZS},
is naturally non-commutative and extends to the non-commutative nonlinear Schr\"odinger
and non-commutative modified Korteweg de Vries equations, see Doikou \textit{et al.\/} \cite{DMS20}.
They also show how the method also naturally extends to nonlocal versions of
these equations in the sense given in Ablowitz and Musslimani~\cite{AM},
i.e.\/ where the nonlocality consists of reverse space-time or
reverse time fields as factors in the nonlinear terms. 
The results herein extend the solution method developed in  
Doikou \textit{et al.\/} \cite{DMS20}, non-trivially,
to the fourth order non-commutative case.

Let us explain P\"oppe's approach in some more detail.
Consider a nonlinear complex matrix-valued partial differential equation
for $g=g(x;t)$ of the form:
\begin{equation*}
\pa_tg=d(\pa)g+\Psi(g,\pa g,\pa^2 g,\ldots),
\end{equation*}
where $\pa=\pa_x$. Here we suppose $d=d(\pa)$ is a constant coefficient polynomial in $\pa$,
while $\Psi$ is a precise homogeneous non-commutative polynomial function of $g$ and
its partial derivatives up to an order two less than the degree of $d$.
Without loss of generality we assume we have subsumed any homogeneous linear terms
in $\Psi$ into the term $d(\pa)g$. Hence $\Psi$ encodes all the nonlinear terms
in the partial differential equation shown.
Setting $\Psi=O$, the zero matrix, consider the corresponding linear partial differential equation
for the complex matrix-valued function $p=p(x;t)$ as follows:
\begin{equation*}
\pa_tp=d(\pa)p. 
\end{equation*}
The quantity $p=p(x;t)$ represents the `scattering data'. The first step in P\"oppe's
approach is to elevate the Marchenko equation to the operator level. 
We construct a Hankel operator $P=P(x,t)$ associated with the scattering data as follows.
We assume $P$ is a Hilbert--Schmidt operator with integral kernel given by
$p=p(y+z+x;t)$ so that for any square-integrable function $\phi$:
\begin{equation*}
  (P\phi)(y;x,t)\coloneqq\int_{-\infty}^0 p(y+z+x;t)\phi(z)\,\mathrm{d}z.
\end{equation*}
Note that $P=P(x,t)$ satisfies the operator differential equation $\pa_tP=d(\pa)P$. 
We then define an associated `data' operator $Q=Q(x,t)$ by $Q\coloneqq P^\dag P$,
where $P^\dag$ is the adjoint operator to $P$. This precise assignment for $Q$ does depend
on the application at hand. For the applications herein we make the choice stated,
guided by Ablowitz \textit{et al.\/} \cite{ARS}. The crucial classical ingredient
is the  Marchenko equation and here, at the operator level, this has the form:
\begin{equation*}
  P=G(\id+Q). 
\end{equation*}
This is a linear Fredholm equation for the operator $G=G(x,t)$.
Hence to recap, the three key ingredients in the first step in P\"oppe's approach
are the: (i) Operator differential equation for $P=P(x,t)$; (ii) Assignment
for the auxiliary data operator $Q=Q(x,t)$ and (iii) Linear Fredholm equation for $G=G(x,t)$.

The second step in P\"oppe's approach, and the major underlying insight,
is the `kernel product rule', in which the Hankel property of $P$ plays a crucial role.
Suppose $F=F(x,t)$ is a Hilbert--Schmidt linear operator with
kernel $f(y,z;x,t)$. Note we assume $f$ depends on the parameters $x$ and $t$.
For example recall from above $P=P(x,t)$ is a Hankel operator with kernel  $p=p(y+z+x;t)$.
Let us denote by $[F]$ the kernel of $F$, i.e.\/ $[F]=f$. 
Now suppose $F=F(x,t)$ and $F'=F'(x,t)$ are Hilbert--Schmidt operators
with kernels continuous in $x$. In addition suppose $H$ and $H'$ are
Hilbert--Schmidt \emph{Hankel} operators with kernels continuously differentiable in $x$. 
Then the fundamental theorem of calculus implies
\begin{equation*}
\bigl[F\pa_x(HH')F'\bigr](y,z;x,t)=[FH](y,0;x,t)[H'F'](0,z;x,t).
\end{equation*}
This is the crucial `kernel product rule' composing the second step in P\"oppe's approach.
More precisely, P\"oppe used the `trace' form of this rule evaluated at $y=z=0$.
We prefer to delay this specialisation until the final step in the procedure.
The kernel product rule is the only property we use
in Doikou \textit{et al.\/} \cite{DMSW20,DMS20} and herein.

The third and final step is to compute $\pa_tG-d(\pa)G$ where
from the linear Fredholm equation above $G=PU$ with $U\coloneqq(\id+Q)^{-1}$.
And we then apply the kernel bracket operator $[\,\cdot\,]$.
The basic calculus property $\pa U=-U(\pa Q)U$ initiates the generation of
nonlinear terms. The goal is then to use \emph{only} the kernel product rule
to establish a `closed form' for the nonlinear terms generated. 
By a `closed form' we mean the terms generated represent
a constant coefficient non-commutative polynomial in
$[G]$, $\pa[G]$, $\pa^2[G]$ and so forth. Hence for example, if
$d(\pa)=\mu_4\pa^4$ and $\mu_4$ is a pure imaginary constant parameter  
then in our main Theorem~\ref{thm:main} we show if $P=P(x,t)$ satisfies
$\pa_tP=d(\pa)P$ and $Q=P^\dag P$, then $[G]$ 
satisfies the non-commutative nonlinear partial differential equation,
\begin{align*}
\pa_t[G]-\mu_4\pa^4[G]
=&\;2\mu_4\Bigl(2\bigl(\pa^2[G]\bigr)[G]^\dag[G]+[G]\bigl(\pa^2[G]^\dag\bigr)[G]+2[G][G]^\dag\bigl(\pa^2[G]\bigr)\\
&\;+\bigl(\pa[G]\bigr)\bigl(\pa[G]^\dag\bigr)[G]+3\bigl(\pa[G]\bigr)[G]^\dag\bigl(\pa[G]\bigr)
+[G]\bigl(\pa[G]^\dag\bigr)\bigl(\pa[G]\bigr)\\
&\;+3[G][G]^\dag[G][G]^\dag[G]\Bigr).
\end{align*}
In the above, $P^\dag$ is the operator adjoint to $P$ while $[G]^\dag$ is the kernel function
corresponding to the complex-conjugate transpose of the kernel function $[G]$.
In the formulation above we have suppressed the parameter dependence of both the
kernels $[G]=[G](y,z;x,t)$ and the kernels of the nonlinear terms---recall the 
kernel product rule above. Hence for example two applications of the kernel
product rule led to the first term on the right which in full should read 
\begin{equation*}
4\mu_4\bigl(\pa^2[G](y,0;x,t)\bigr)[G]^\dag(0,0;x,t)[G](0,z;x,t),
\end{equation*}
and so forth for the other cubic terms. Four applications of the
kernel product rule generated the quintic term whose left and right factors
should have parameter dependencies matching those of the corresponding factors
in the cubic term above, and whose three central factors
should have the parameter dependence $(0,0;x,t)$. By a standard convention
we invoke, these dependencies are implied in the non-commutative equation above.
We now emphasise that we can set $y=z=0$ throughout so that all the terms have
the parameter dependence $(0,0;x,t)$. This generates the 
non-commutative fourth order quintic nonlinear Schr\"odinger equation.
Furthermore the solution to this equation is generated as follows.
We solve the \emph{linear} partial differential equation, namely $\pa_tp=\mu_4\pa^4p$.
This can be achieved analytically.
The solution function $p$ generates the kernel of the Hankel operator $P=P(x,t)$.
We set $Q=P^\dag P$, this involves computing an integral whose integrand is a known function.
We can then compute the solution to the non-commutative fourth order quintic nonlinear Schr\"odinger equation for
$[G]=[G](y,z;x,t)$ shown above by solving the \emph{linear} Fredholm equation $P=G(\id+Q)$ for $G$.
Hence the quintic nonlinear Schr\"odinger above is linearsiable and thus integrable in this sense.

As another example, consider the case of a non-commutative fourth order quintic nonlinear Schr\"odinger equation
with a nonlocal nonlinearity as follows. Suppose the Hankel operator $P=P(x,t)$ satisfies
the same partial differential equation $\pa_tP=\mu_4\pa^4P$ as in the example just above
with $\mu_4$ a pure imaginary constant parameter. However, we now specify that $Q=\tilde{P}P$
where the operator $\tilde{P}=\tilde{P}(x,t)$ is given by $\tilde{P}(x,t)=P^{\mathrm{T}}(-x,-t)$,
where $P^{\mathrm{T}}$ is the operator whose matrix kernel is the transpose of the matrix kernel
corresponding to $P$. Then $[G]$ satisfies an analogous non-commutative nonlinear partial differential equation
to that shown above, except that the terms $[G]^\dag$ on the right-hand side are replaced by
$[\tilde{G}]$ where $\tilde{G}=\tilde{G}(x,t)$ is given by $\tilde{G}(x,t)=G^{\mathrm{T}}(-x,-t)$.
As in the last example, $[G]=[G](y,z;x,t)$, and we can set $y=z=0$ throughout so that
all the terms have the parameter dependence $(0,0;x,t)$. This generates the 
reverse space-time nonlocal non-commutative fourth order quintic nonlinear Schr\"odinger equation;
see Example~\ref{ex:reversespace-time} in the main text.
The solution to this equation can be generated
in an analogous manner to that described in the example just above,
demonstrating the equation is linearisable and thus integrable.

In Doikou \textit{et al.\/} \cite{DMSW20} with $d(\pa)=\mu_2\pa^2$ and $\mu_2$
a constant pure imaginary parameter, we generated the solution to the
nonlinear Schr\"odinger equation in this way. With $d(\pa)=\mu_3\pa^3$
and $\mu_3$ a real parameter, and a slight modification of the procedure above,
we generated solutions to the Korteweg de Vries equation.
Then in Doikou \textit{et al.\/} \cite{DMS20} we generalised this approach
to the non-commutative setting and also generated solutions to the
non-commutative modified Korteweg de Vries in this way from $d(\pa)=\mu_3\pa^3$.
Note, with the solution procedure described above, the specific choices
of $d=d(\pa)$ indicated, generate precise non-commutative polynomial functions $\Psi$
representing the nonlinear terms. In preceding work, Beck \textit{et al.\/}
\cite{BDMSI,BDMSII} assume the kernels $p=p(y,z;t)$ and $q=q(y,z;t)$
associated with the operators $P$ and $Q$ satisfy a coupled
pair of linear partial differential equations. They show the kernel $[G]=[G](y,z;t)$
associated with the operator $G$ solving the linear Fredholm equation $P=G(\id+Q)$
satisfies a Riccati partial differential equation which can be interpreted as
a nonlocal nonlinear partial differential equation. For example
Beck \textit{et al.\/} \cite{BDMSII} generate solutions to the following
nonlocal Korteweg de Vries equation for $[G]=[G](y,z;t)$ using this approach:
\begin{equation*}
\pa_t[G](y,z;t)-\pa^3_y[G](y,z;t)=\int_{\R}[G](y,\xi;t)\bigl(\pa_\xi[G](\xi,z;t)\bigr)\,\rd\xi.
\end{equation*}
The nonlocal nonlinearity is the realisation at the kernel level of a linear operator product, and
the kernel product rule is not used. All the nonlinear flows in Beck \textit{et al.\/}
\cite{BDMSI,BDMSII} and Doikou \textit{et al.\/} \cite{DMSW20,DMS20} are shown to be
Grassmannian flows. Indeed the theory in Doikou \textit{et al.\/} \cite[Section~2.3]{DMS20}
establishes the flow generated in our main Theorem~\ref{thm:main} is also a Grassmannian flow. 

In actuality, we consider an inflated coupled linear system, one which includes
the linear partial differential equation $\pa_tP=d(\pa)P$ for $P=P(x,t)$,
but more generally assigns $Q\coloneqq\tP P$ where $\tP=\tP(x,t)$ is a linear operator analogous
to $P$ satisfying an associated linear partial differential equation $\pa_t\tP=\tilde{d}(\pa)\tP$.
Here $\tilde{d}=\tilde{d}(\pa)$ is a constant coefficient
polynomial in $\pa$ analogous to $d(\pa)$, of the same degree.
We correspondingly assign $\tQ\coloneqq P\tP$.
And finally in addition to $P=G(\id+Q)$ we now include the analogous linear Fredholm equation
$\tP=\tG(\id+\tQ)$. See Definition~\ref{def:linearoperatorsystem} for the inflated linear system.
The inflated system also naturally generates a Grassmannian flow; see Doikou \textit{et al.\/} \cite{DMS20}.
Naturally we must assume the complex matrix-valued kernels of $P$ and $\tP$ are commensurate
so that $Q$ and $\tQ$ make sense. Consequently with suitable restrictions on $\tilde{d}(\pa)$
we can choose $\tP=P^\dag$ as above. Or for example, we can also consistently choose
$\tP(x,t)=P^{\mathrm{T}}(-x,-t)$ where $P^{\mathrm{T}}$ is the linear operator whose kernel
is the transpose of the kernel for $P$. In this case $\tG(x,t)=G^{\mathrm{T}}(-x,-t)$
and we generate a non-commutative quintic equation like that above with $[G]^\dag$
replaced by $G^{\mathrm{T}}(-x,-t)$ everywhere. Thus the corresponding
reverse space-time nonlocal non-commutative quintic equation, in the sense
of Ablowitz and Musslimani~\cite{AM}, is linearisable. 

We also develop a numerical method for solving such integrable systems,
first explored in Doikou \textit{et al.} \cite{DMSW20}. 
The numerical method applies to any of the integrable systems
such as the non-commutative nonlinear Schr\"odinger and modified Korteweg de Vries equations considered
in Doikou \textit{et al.\/} \cite{DMS20}, as well as the non-commutative
fourth order quintic equation above, or indeed, the more
general non-commutative fourth order quintic equation which we establish
is linearisable and integrable as our main result in Theorem~\ref{thm:main}.
We can analytically solve the linearised partial differential system
for the Hankel operator $P=P(x,t)$ or indeed its corresponding kernel $p=p(x,t)$,
in Fourier space. This is because $p$ satisfies the linear partial differential equation
$\pa_t=d(\pa)p$, where $d(\pa)$ is a constant coefficient polynomial in $\pa$.
Hence in principle we can evaluate $p=p(x,t)$ at any given time $t>0$, or in practice
we can represent it to any degree of accuracy determined by the finite number
of Fourier modes we choose to represent the Fourier series of $p=p(x,t)$,
on a truncated domain.
We can then determine the kernel function $q=q(y,z;x,t)$ corresponding to 
the associated data operator $Q=Q(x,t)$ by evaluating $Q=P^\dag P$ by computing,
\begin{equation*}
q(y,z;x,t)=\int_{-\infty}^0p^\dag(y+\zeta+x;t)p(\zeta+z+x;t)\,\rd\zeta.
\end{equation*}
Note $P^\dag$ denotes the adjoint operator to $P$, while for the kernel function $p^\dag$
denotes the complex conjugate transpose to the complex matrix-valued function $p$.
Using our representation for $p$ we can approximate the integral on the right-hand side
to compute an approximation for $q$. To compute the operator $G$ we need to solve
the Fredholm equation $P=G(\id+Q)$, i.e.\/ if $g=g(y,z;x,t)$ is the kernel corresponding
to $G$, we numerically solve the Fredholm integral equation
\begin{equation*}
p(y+z+x;t)=g(y,z;x,t)+\int_{-\infty}^0g(y,\zeta;x,t)q(\zeta,z;x,t\,\rd\zeta.
\end{equation*}
We achieve this by approximating the integral on the right-hand side by a Riemann sum
and solving the resulting large linear algebraic system of equations.
We then set $y=z=0$ to determine $g(0,0;x,t)$, the kernel correspondng to $[G](0,0;x,t)$.
The method works for any given initial data function $p_0=p_0(x)$ for which $p(x,0)=p_0(x)$.
In principle, given any initial data function $g_0=g_0(x)$ for which $g(0,0;x,0)=g_0(x)$
we could compute $p(x,0)$ via `scattering' methods, however we do not implement this here.
In general the overall numerical procedure we implement is straightforward, and in practice,
appears to be robust. 

In addition to the series of papers by P\"oppe mentioned above,
the work herein was also motiviated by Ablowitz \textit{et al.\/} \cite{ARS},
Dyson \cite{Dyson} and McKean \cite{McKean}. We also mention in this
context Ercolani and McKean \cite{EM} and Mumford~\cite[p.~3.239]{Mumford}.
The Marchenko equation is central not only to P\"oppe's approach,
but also to that of Fokas and Ablowitz~\cite{FA},
Nijhoff \textit{et al.\/} \cite{NQLC} and the Zakharov--Shabat scheme \cite{ZS,ZS2}.
Details of Fokas' unified transform method can be found for example in Fokas and Pelloni~\cite{FP}.
Hankel operators have received a lot of recent attention, see Grudsky and Rybkin \cite{GR1, GR2},
Grellier and Gerard~\cite{Gerard} and Blower and Newsham \cite{BN}.
Nonlocal integrable systems have also received a lot of recent attention,
see Ablowitz and Musslimani \cite{AM}, Fokas \cite{F2016}, Grahovski, Mohammed and Susanto \cite{GMS}
and G\"urses and Pekcan~\cite{GP-NLNLS-mKdV,GP-NLmKdV,GP-NLNLS,GP-NLKdV}.
The scalar reverse time nonlocal nonlinear Schr\"odinger equation is $PT$ symmetric,
and was derived in Ablowitz and Musslimani~\cite{AM2013} ``with physical intuition''.
Indeed, Ablowitz and Musslimani~\cite{AM2019} establish
``an important physical connection between the recently discovered
nonlocal integrable reductions of the AKNS system and physically interesting equations''.
Further, quoting from Gerdjikov and Saxena~\cite{GS}, ``nonlocal, nonlinear equations arise
in a variety of physical contexts ranging from hydrodynamics to optics to condensed matter
and high energy physics''. See Lou and Huang~\cite{LH} for a derivation of
the nonlocal `Alice--Bob Korteweg de Vries' system from a system of equations modelling
atmospheric dynamics.
A local non-commutative fourth order quintic nonlinear Schr\"odinger equation corresponding to that above can be
found in Nijhoff \textit{et al.\/} \cite[eq.~B.4a]{NQLC} who establish integrability
via the method pioneered by Fokas and Ablowitz~\cite{FA}.
Indeed Nijhoff \textit{et al.\/} \cite[eq.~B.5a]{NQLC} also present
the non-commutative fifth order quintic nonlinear Schr\"odinger equation.
For further early work on non-commutative integrable systems, see in addition, 
Manakov \cite{M1974}, Fordy and Kulisch \cite{FK} and Ablowitz \textit{et al.\/} \cite{APT}.
For more recent work on the multi-component nonlinear Schr\"odinger equation, as well as its discretisation, see 
Ablowitz \textit{et al.\/} \cite{APT}, Degasperis and Lombardo \cite{DL2}
and Pelinovksy~\cite{Pelinovsky}. We remark we do not utilise a Lax pair $L\upsilon=\lambda\upsilon$
and $\pa_t\upsilon=D\upsilon$ for the auxiliary function $\upsilon$ and spectral parameter $\lambda$,
and require compatability. Here we use the linearised evolution equation $\pa_t\upsilon=D\upsilon$
and require $\upsilon$ to have the Hankel property.

General higher order nonlinear Schr\"odinger equations have recently received a lot of attention;
see Karpman \cite{K}, Karpman and Shagalov \cite{KS},
Ben--Artzi \textit{et al.\/} \cite{B-AKS}, Fibich \textit{et al.\/} \cite{FIP},
Pausader~\cite{Pausader}, Boulenger and Lenzmann~\cite{BL},
Kwak~\cite{Kwak}, Oh and Wang~\cite{OW} and Posukhovskyi and Stefanov~\cite{PosukhovskyiStefanov}.
More specifically though, in our main result in Theorem~\ref{thm:main} we
establish integrability for a more general system of equations than that
shown for $[G]$ above. The more general system also includes standard
nonlinear Schr\"odinger dispersion and nonlinear terms, both with
the scalar constant pure imaginary factor $\mu_2$, as well
as the next order terms in the nonlinear Schr\"odinger hierarchy,
namely standard non-commutative modified Korteweg de Vries third order dispersion and nonlinear terms,
both with the scalar constant real factor $\mu_3$.
The commutative form of this more general equation, involving
the parameters $\mu_2$, $\mu_3$ and $\mu_4$ have recently
found applications as models for short pulse propagation in
optical fibres as follows; see Kang \textit{et al.} \cite{KangXiaMa} and Agrawal~\cite{Agrawal}.
The fundamental nonlinear Schr\"odinger equation itself, corresponding to $\mu_3=\mu_4=0$
describes the propagation of picosecond pulses in an optical fibre.
The case when $\mu_4=0$ (only) corresponds to the Hirota equation which describes 
the propagation of femtosecond soliton pulses in the mono-mode optical fibres,
see Demiray \textit{et al.} \cite{Demiray}, Mihalache \textit{et al.} \cite{Mihalache} and 
Nakkeeran~\cite{Nakkeeran}. The case when with $\mu_2$ and $\mu_4$
pure imaginary, and $\mu_3$ real, with all three non-zero describes ``ultrashort
optical-pulse propagation in a long-distance, high-speed optical fibre transmission system'', 
see again Kang \textit{et al.} \cite{KangXiaMa} or Guo \textit{et al.} \cite{GuoHaoGu}.
Also see Wang \textit{et al.} \cite{WangPorsezianHe} who consider the case $\mu_3=0$ (only).
Attosecond pulses in an optical fibre are described by a fifth order nonlinear Schr\"odinger equation. 
Indeed Kang \textit{et al.} \cite{KangXiaMa} consider and eighth order nonlinear Schr\"odinger equation as a model
for ultrashort pulse propagation. The nonlinear Schr\"odinger hierarchy up to and including order eight
can be found in Matveev and Smirnov~\cite{MatveevSmirnov} who consider 
multi-rogue wave solutions. For applications to pulses in erbium-doped fibres
modelled by the higher order nonlinear Schr\"odinger equations above coupled to a Maxwell--Bloch
system, see Guan \textit{et al.} \cite{GuanTianZhenWangChai}, Guo \textit{et al.} \cite{GuoHaoGu},
Ren \textit{et al.} \cite{REnYangLiuXuYang}, Wang \textit{et al.} \cite{WLQ},
Wang \textit{et al.} \cite{WangWuZhang} and Wang \textit{et al.} \cite{WangGaoSuZuo}.

To summarise, what is new in this paper is we: 
\begin{enumerate}
\item[(i)] Give a direct proof, based on the Hankel operator approach of Ch. P\"oppe,
that a generalised non-commutative fourth order quintic nonlinear Schr\"o-dinger equation is linearisable
and thus integrable (Theorem~\ref{thm:main});

\item[(ii)] Prove reverse space-time and reverse time nonlocal versions,
in the sense of Ablowitz and Musslimani~\cite{AM},
of the non-commutative fourth order quintic nonlinear Schr\"odinger equation,
are also linearisable and thus integrable. These results are specialisations of the
approach utilised to establish (i);

\item[(iii)] Develop a numerical method to accurately evaluate the solution to
such integrable systems at any given time, based on the analytical linearisation approach in (i).
The numerical method involves the analytical solution
of the linearised partial differential system at the time given, 
computing an associated data function and then numerically solving
a linear Fredholm equation to generate the solution at that time.
The method is straightforward, and in practice appears to be robust.
\end{enumerate}

Our paper is organised as follows. In Section~\ref{sec:preliminaries}
we introduce the notation and key concepts and identities we need to
prove our main result. The latter is presented and proved in Section~\ref{sec:main}.
We demonstrate our numerical method based on the analytical approach above
in Section~\ref{sec:numericalsimulations}. We include some insights on,
and comments on future directions for, the work herein in the final
Discussion Section~\ref{sec:discussion}.

\section{Preliminaries}\label{sec:preliminaries}
We consider Hilbert--Schmidt integral operators which depend on both a 
spatial parameter $x\in\mathbb{R}$ and a time parameter $t\in[0,\infty)$.
Throughout $\pa_t$ represents the partial derivative with respect to
the time parameter $t$ while $\pa=\pa_x$ represents the partial derivative
with respect to the spatial parameter $x$.
Hilbert--Schmidt operators are representable in terms of square-integrable kernels.
Hence for a given Hilbert--Schmidt operator $F=F(x,t)$,
there exists a square-integrable kernel $f=f(y,z;x,t)$ such that
for any square-integrable function $\phi$,
\begin{equation*}
  (F\phi)(y;x,t) = \int_{-\infty}^0 f(y,z;x,t)\phi(z)\,\mathrm{d}z.
\end{equation*}
\begin{definition}[Kernel bracket]
With reference to the operator $F$ just above, we use the \emph{kernel bracket} notation
$[F]$ to denote the kernel of $F$:
\begin{equation*}
  [F](y,z;x,t) \coloneqq f(y,z;x,t).
\end{equation*}
We often drop the dependencies and simply write $[\,\cdot\,]$.
\end{definition}
Of critical importance throughout this paper is a class of integral operators known as
Hankel operators. We consider Hankel operators which depend on a parameter $x$ as follows. 
\begin{definition}[Hankel operator with parameter]\label{def:Hankel}
We say a given time-dependent Hilbert--Schmidt operator $H$
with corresponding square-integrable kernel $h$ is \emph{Hankel} or \emph{additive}
with parameter $x\in\R$ if its action, for any square-integrable function $\phi$, is given by
\begin{equation*}
  (H\phi)(y;x,t) \coloneqq \int_{-\infty}^0 h(y+z+x;t)\phi(z)\,\mathrm{d}z.
\end{equation*}
\end{definition}
Hankel operators of this form are the starting point for P\"oppe's approach; see
P\"oppe~\cite{P83,P84} and Doikou \textit{et al.\/} \cite{DMSW20,DMS20}.
As mentioned in the introduction there is a crucial kernel product rule we rely on throughout.
This is as follows. We include the proof from Doikou \textit{et al.\/} \cite{DMSW20,DMS20} for completeness.
\begin{lemma}[Kernel product rule]\label{lemma:kernelproductrule}
Assume $H,H'$ are Hilbert--Schmidt \emph{Hankel} operators with parameter $x$ and $F,F'$ are Hilbert--Schmidt operators.
Assume further that the corresponding kernels of $F$ and $F'$ are continuous and
of $H$ and $H'$ are continuously differentiable. Then, the following \emph{kernel product rule} holds,
\begin{equation*}
  [F\pa(HH')F'](y,z;x) = [FH](y,0;x)[H'F'](0,z;x).
\end{equation*}
\end{lemma}
\begin{proof}
We use the fundamental theorem of calculus and Hankel properties of $H$ and $H'$.
Let $f$, $h$, $h'$ and $f'$ denote the integral kernels of $F$, $H$, $H'$ and $F'$ respectively.
By direct computation $[F\pa_x(HH')F'](y,z;x)$ equals
\begin{align*}
&\int_{\R_-^3}
f(y,\xi_1;x)\pa_x\bigl(h(\xi_1+\xi_2+x)h^\prime(\xi_2+\xi_3+x)\bigr)
f^\prime(\xi_3,z;x)\,\rd \xi_3\,\rd \xi_2\,\rd \xi_1\\
&=\int_{\R_-^3}
f(y,\xi_1;x)\pa_{\xi_2}\bigl(h(\xi_1+\xi_2+x)h^\prime(\xi_2+\xi_3+x)\bigr)
f^\prime(\xi_3,z;x)\,\rd \xi_3\,\rd \xi_2\,\rd \xi_1\\
&=\int_{\R_-^2}
f(y,\xi_1;x)h(\xi_1+x)h^\prime(\xi_3+x)f^\prime(\xi_3,z;x)\,\rd \xi_3\,\rd \xi_1\\
&=\int_{\R_-}f(y,\xi_1;x)h(\xi_1+x)\,
\rd \xi_1\cdot\int_{\R_-}h^\prime(\xi_3+x)f^\prime(\xi_3,z;x)\,\rd \xi_3\\
&=\bigl([FH](y,0;x)\bigr)\bigl([H'F'](0,z;x)\bigr),
\end{align*}
giving the result. 
\end{proof}
This kernel bracket operator and product rule above originates from
the work of P\"oppe in \cite{P83, P84, P-KP} and Bauhardt and P\"oppe~\cite{BP-ZS}.
We record in the following lemma some identities for $W\coloneqq (\mathrm{id}+F)^{-1}$
which are useful later on. We assume $F$ depends on a parameter.
Similar results are derived by P\"oppe \cite{P83, P84}.
\begin{lemma}[Inverse operator identities]\label{lemma:invOp}
Suppose the operator $F$ depends on a parameter with respect
to which we wish to compute derivatives. Further suppose $W\coloneqq (\mathrm{id}+F)^{-1}$ exists.
Then the following identities hold:
\begin{enumerate}
\item[(i)] $\id-W=WF=FW$;
\item[(ii)] $\pa W=-W(\pa F)W$;
\item[(iii)] $\pa W=-(\pa W)F-W(\pa F)=-(\pa F)W-F(\pa W)$;
\item[(iv)] $\pa^2W=-2(\pa W)(\pa F)W-W(\pa^2F)W$;
\item[(v)] $\pa^3W=-3(\pa^2 W)(\pa F)W-3(\pa W)(\pa^2F)W-W(\pa^3 F)W$;
\item[(vi)] $\pa^4W=-4(\pa^3 W)(\pa F)W-6(\pa^2 W)(\pa^2 F)W-4(\pa W)(\pa^3F)W-W(\pa^4 F)W$.
\end{enumerate}
\end{lemma}
\begin{proof}
The first identity is straightforward. The others follow by successively
differentiating $\id-W=WF$ and finally using $W\coloneqq(\id+Q)^{-1}$ in each form.  
\end{proof}
We have the following corollary to Lemma~\ref{lemma:invOp},
which can also be found in Doikou \textit{et al.}~\cite{DMS20}.
\begin{corollary}\label{cor:invId}
Suppose we set $F\coloneqq Q$ with $Q=\tP P$ and $U\coloneqq (\mathrm{id}+Q)^{-1}$
so $U=(\mathrm{id}+\tP P)^{-1}$.  Assume $U$ exists.
Then $U$ satisfies properties (i)--(vi) in Lemma~\ref{lemma:invOp}.
Further suppose we set $F\coloneqq\tilde Q$ with $\tilde Q=P\tP$ and $V\coloneqq(\mathrm{id}+\tilde Q)^{-1}$.
Assume $V$ exists. Then similarly $V$ satisfies properties (i)--(vi) in Lemma~\ref{lemma:invOp}.
We note $PU^{-1}=V^{-1}P$, so we have $VP=PU$. Similarly, we have $U\tP=\tP V$.  
\end{corollary}
The following key identities also prove useful throughout the proof of
our main result in Section~\ref{sec:main}. To keep our statements succinct
hereafter we use the following \emph{notation convention}. For the
kernel product rule introduced in Lemma~\ref{lemma:kernelproductrule},
for the terms on the right we simply write $[FH][H'F']$, where it is
understood the left factor is evaluated at $(y,0;x,t)$ and the right
factor is evaluated at $(0,z;x,t)$. When there are three
factors in the product, such as for the case $[F_1\pa(H_1H_1')F_1'F_2\pa(H_2H_2')F_2']$
where $H_1$, $H_1'$, $H_2$ and $H_2'$ are Hankel operators, we write
\begin{equation*}
[F_1H_1][H_1'F_1'F_2H_2][H_2'F_2'],
\end{equation*}
where it is understood the left and right factors are evaluated at $(y,0;x,t)$ and
$(0,z;x,t)$ respectively, while the middle factor is evaluated at $(0,0;x,t)$.
This is just a direct consequence of successively applying the kernel product rule.
For higher degree products of the form just above, again the left and right factors
are always evaluated at $(y,0;x,t)$ and $(0,z;x,t)$ respectively,
while all the middle factors are evaluated at $(0,0;x,t)$.
\begin{lemma}[Key identities] \label{lemma:keyidentities} 
Assume the Hilbert--Schmidt operators $P$, $\tP$, $G$, $\tG$
and trace class operators $Q\coloneqq\tP P$, $\tQ\coloneqq P\tP$ all depend on a parameter $x$ and
are related by $P=G(\id+Q)$ and $\tP=\tG(\id+\tQ)$. Assume $P$ and $\tP$
are Hankel operators as in Definition~\ref{def:Hankel} and also the inverse operators
$U\coloneqq(\id+Q)^{-1}$ and $V\coloneqq(\id+\tQ)^{-1}$ exist. 
Then we have the following identities:
\begin{enumerate}
\item[(i)] $\pa[PU\tP]=[G][\tG]$;
\item[(ii)] $\pa\bigl[PU(\pa\tP)\bigr]=[G]\pa[\tG]+[G][\tG][PU\tP]$;
\item[(iii)] $\pa\bigl[PU(\pa^2\tP)\bigr]=[G]\pa^2[\tG]+2\bigl([G][\tG]\bigr)^2
              +[G]\bigl(\pa[\tG]\bigr)[PU\tP]-[G][\tG]\bigl[\pa(PU)\tP\bigr]$;
\end{enumerate}
and by partial differentiation that,
\begin{enumerate}
\item[(iv)] $\pa^2[PU\tP]=\pa\bigl([G][\tG]\bigr)$;
\item[(v)] $\pa^2\bigl[PU(\pa\tP)\bigr]=\pa\bigl([G]\pa[\tG]\bigr)
            +\bigl([G][\tG]\bigr)^2+\pa\bigl([G][\tG]\bigr)[PU\tP]$ 
\end{enumerate}
\end{lemma}
\begin{proof}
First, by differentiating the formula $\id-V=PU\tP$ partially with respect to $x$ we see,
$\bigl[\pa(PU\tP)\bigr]=-[\pa V]=\bigl[V\pa(P\tP)V\bigr]=[VP][\tP V]=[G][\tG]$, giving (i).

Second, for (ii), using the product rule and $\id-V=PU\tP$ we find:
\begin{align*}
\pa(PU(\pa\tP))=&\;\pa(PU)(\pa\tP)+PU(\pa^2\tP)\\
=&\;\pa(PU)(\pa\tP)V+PU(\pa^2\tP)V+\pa(PU)(\pa\tP)PU\tP+PU(\pa^2\tP)PU\tP\\
=&\;\pa(VP)(\pa\tP)V+VP(\pa^2\tP)V+\pa(PU)\pa(\tP P)U\tP+PU\pa\bigl((\pa\tP) P\bigr)U\tP\\
&\;-\pa(PU)\tP(\pa P)U\tP-PU(\pa\tP)(\pa P)U\tP.
\end{align*}
The last two terms on the right combine to become $-(\pa V)(\pa P)\tP V$ since we know
$\pa(PU\tP)=-\pa V$ and $U\tP=V\tP$.
Substituting this result into the expression above and combining $-(\pa V)(\pa P)U\tP=-(\pa V)(\pa P)\tP V$
with the first two terms on the right and then applying the kernel bracket operator we observe:
\begin{align*}
\bigl[\pa(PU(\pa\tP))\bigr]
=&\;[V\pa(P(\pa\tP))V]+[(\pa V)\pa(P\tP)V]\\
  &\;+[\pa(PU)\tP][PU\tP]+[PU(\pa\tP)][PU\tP]\\
=&\;[VP][(\pa\tP)V]+[(\pa V)P][\tP V]
  +[\pa(PU\tP)][PU\tP].
\end{align*}
Note we have
\begin{equation*}
[(\pa\tP)V]=[\pa(\tP V)]-[\tP(\pa V)]=\pa[\tP V]+[\tP V\pa(P\tP)V]=\pa[\tP V]+[\tP VP][\tP V],
\end{equation*}
and $[(\pa V)P]=-[V\pa(P\tP)VP]=-[VP][\tP VP]$.
Substituting these into the right-hand side, cancelling like terms and using (i)
gives (ii).

Third, we prove (iii) using a similar strategy to that we used for (ii),
Using the product rule and $\id-V=PU\tP$ we find:
\begin{align*}
\pa\bigl(PU(\pa^2\tP)\bigr)=&\;\pa(PU)(\pa^2\tP)+PU(\pa^3\tP)\\
=&\;\pa(PU)(\pa^2\tP)V+PU(\pa^3\tP)V+\pa(PU)(\pa^2\tP)PU\tP+PU(\pa^3\tP)PU\tP\\
=&\;\pa(VP)(\pa^2\tP)V+VP(\pa^3\tP)V+\pa(PU)\pa\bigl((\pa \tP)P\bigr)U\tP\\ 
&\;+PU\pa\bigl((\pa^2\tP) P\bigr)U\tP-\pa(PU)(\pa\tP)(\pa P)U\tP-PU(\pa^2\tP)(\pa P)U\tP\\
=&\;(\pa V)P(\pa^2\tP)V+V\pa\bigl(P(\pa^2\tP)\bigr)V+\pa(PU)\pa\bigl((\pa \tP)P\bigr)U\tP\\
&\;+PU\pa\bigl((\pa^2\tP) P\bigr)U\tP-\pa\bigl(PU(\pa\tP)\bigr)(\pa P)U\tP.
\end{align*}
If we now apply the kernel bracket operator to the expression above, and combine
the third and fourth terms on the right, use that $\pa V=-V\pa(P\tP)V$ and use the
kernel product rule, we obtain:
\begin{align*}
\pa\bigl[PU(\pa^2\tP)\bigr]
=&\;-[G]\bigl[\tP VP(\pa^2\tP)V\bigr]+[G]\bigl[(\pa^2\tP)V\bigr]+\bigl(\pa\bigl[PU(\pa \tP)\bigr]\bigr)[PU\tP]\\
&\;-\bigl[\pa\bigl(PU(\pa\tP)\bigr)(\pa P)U\tP\bigr].
\end{align*}
Note for the second factor of the second term on the right just above we have:
\begin{align*}
\bigl[(\pa^2\tP)V\bigr]
=&\;\pa^2[G]-2\bigl[(\pa\tP)(\pa V)\bigr]-\bigl[\tP(\pa^2 V)\bigr]\\
\mathrm{(a)}~=&\;\pa^2[G]+2\bigl[(\pa\tP)V\pa(P\tP)V\bigr]+2\bigl[\tP(\pa V)\pa(P\tP)V\bigr]+\bigl[\tP V\pa^2(P\tP)V\bigr]\\
\mathrm{(b)}~=&\;\pa^2[G]+2\bigl[(\pa\tP)VP\bigr][\tG]
+2\bigl[\tP(\pa V)P\bigr][\tG]+2\bigl[\tP V\pa\bigl((\pa P)\tP\bigr)V\bigr]\\
&\;-\bigl[\tP V(\pa^2 P)\tP V\bigr]+\bigl[\tP VP(\pa^2\tP)V\bigr]\\
\mathrm{(c)}~=&\;\pa^2[G]+2\pa\bigl[\tP VP\bigr][\tG]-\bigl[\tP V(\pa^2 P)\tP V\bigr]+\bigl[\tP VP(\pa^2\tP)V\bigr]\\
\mathrm{(d)}~=&\;\pa^2[G]+2[\tG][G][\tG]-\bigl[\tP V(\pa^2 P)\tP V\bigr]+\bigl[\tP VP(\pa^2\tP)V\bigr],
\end{align*}
where we used the following, we used the: (a) Identities (i) and (iv) from the inverse operator
Lemma~\ref{lemma:invOp} with $W=V$; (b) Kernel bracket product rule, added and subtracted
the term $[\tP V(\pa^2 P)\tP V]$ and separated the final term from the previous line
as shown; (c) Kernel bracket product rule applied to the fourth term on the right from the
previous line, and combined the resulting term with the other terms with the factor $2$ shown;
and (d) Identity $[\pa(\tP VP)]=-[\pa U]=[U\pa(\tP P)U]=[G][\tG]$.

We now insert this expression for $[(\pa^2\tP)V]$ into the second term
on the right in the expression for $\pa[PU(\pa^2\tP)]$ just above,
cancelling like terms, this gives
\begin{align*}
\pa\bigl[PU(\pa^2\tP)\bigr]
=&\;[G]\pa^2[G]+2\bigl([G][\tG]\bigr)^2-[G]\bigl[\tP V(\pa^2 P)\tP V\bigr]\\
&\;+\bigl(\pa\bigl[PU(\pa \tP)\bigr]\bigr)[PU\tP]-\bigl[\pa\bigl(PU(\pa\tP)\bigr)(\pa P)U\tP\bigr]\\
\mathrm{(e)}~=&\;[G]\pa^2[G]+2\bigl([G][\tG]\bigr)^2+[G]\bigl(\pa[G]\bigr)[PU\tP]+[G][\tG][PU\tP]^2\\
&\;-[G]\bigl[\tP V(\pa^2 P)\tP V\bigr]
-\bigl[\pa\bigl(PU(\pa\tP)\bigr)(\pa P)U\tP\bigr]\\
\mathrm{(f)}~=&\;[G]\pa^2[G]+2\bigl([G][\tG]\bigr)^2+[G]\bigl(\pa[G]\bigr)[PU\tP]-[G][\tG]\bigl[P(\pa U)\tP\bigr]\\
&\;-[G]\bigl[\tP V(\pa^2 P)\tP V\bigr]-\bigl[\pa\bigl(PU(\pa\tP)\bigr)(\pa P)U\tP\bigr],
\end{align*}
where in (e) we used result (ii) and in (f) we used $[PU\tP]^2=[PU\pa(\tP P)U\tP]=-[P(\pa U)\tP]$.
Consider the two terms on the final line on the right-hand side just above. Reversing the kernel product rule
on the first term and using $V\pa(P\tP)V=-\pa V=\pa(PU\tP)$ these two terms become
\begin{align*}
-\bigl[\pa(PU\tP)(\pa^2 P)\tP V\bigr]-\bigl[\pa&\bigl(PU(\pa\tP)\bigr)(\pa P)U\tP\bigr]\\
=&\;-\bigl[\pa(PU)\tP(\pa^2 P)U\tP\bigr]-\bigl[PU(\pa\tP)(\pa^2 P)U\tP\bigr]\\
&\;-\bigl[\pa(PU)(\pa\tP)(\pa P)U\tP\bigr]-\bigl[PU(\pa^2\tP)(\pa P)U\tP\bigr]\\
=&\;-\bigl[\pa(PU)\pa\bigl(\tP(\pa P)\bigr)U\tP\bigr]-\bigl[PU\pa\bigl((\pa\tP)(\pa P)\bigr)U\tP\bigr]\\
=&\;-\bigl[\pa(PU)\tP\bigr]\bigl[(\pa P)U\tP\bigr]-\bigl[PU(\pa\tP)\bigr]\bigl[(\pa P)U\tP\bigr]\\
=&\;-[G][\tG]\bigl[(\pa P)U\tP\bigr].
\end{align*}
Substituting this back into the last expression for $\pa[PU(\pa^2\tP)]$ just above and
combining the terms $[(\pa P)U\tP]+[P(\pa U)\tP]=[\pa(PU)\tP]$ gives result (iii).

Fourth, results (iv) and (v) follow directly by respectively differentiating (i) and (ii)
partially with respect to $x$, and using result (i) itself to help establish (v).
\end{proof}
We now outline the coupled linear operator system that underlies
the non-commutative fourth order quintic nonlinear Schr\"odinger equation we consider herein.
Solutions of this coupled linear system generate solutions to the target nonlinear
local and nonlocal partial differential equations. 
\begin{definition}[Linear operator system]\label{def:linearoperatorsystem}
Suppose the Hilbert--Schmidt linear operators $P$, $\tP$, $Q$, $\tQ$, $G$ and $\tG$
satisfy the coupled linear system of equations:
\begin{equation*}
\begin{aligned}
\pa_tP&=\mu_2\pa^2 P+\mu_3\pa^3 P+\mu_4\pa^4 P,\\
Q&=\tP P,\\
P&=G(\id+Q),
\end{aligned}
~\quad\text{and}~\quad
\begin{aligned}
\pa_t{\tP}&=\tilde{\mu}_2\pa^2 \tP+\tilde{\mu}_3\pa^3\tP+\tilde{\mu}_4\pa^4\tP,\\
\tQ&=P\tP,\\
\tP&=\tG(\id+\tQ).
\end{aligned}
\end{equation*}
where the constant parameters $\mu_2,\mu_3,\mu_4,\tilde{\mu}_2,\tilde{\mu}_3,\tilde{\mu}_4\in\mathbb C$.
Naturally we suppose the matrix kernels of $P$ and $\tP$ are commensurate so $Q$ and $\tQ$ are well-defined.
\end{definition}
We say the system of equations prescribing $G$ and $\tG$ in this definition is linear.
This is because, to determine $G$ and $\tG$, we must first solve the linear partial differential
equations prescribing the flows of $P$ and $\tP$. Then second we determine $Q$ and $\tQ$ directly
from $P$ and $\tP$ without having to solve another equation. And finally third, we determine
$G$ and $\tG$ by solving the linear Fredholm equations shown.

The parameters $\mu_j$ for $j=2,3,4$ are in general arbitrary complex numbers.
However we set $\tilde{\mu}_j=(-1)^{j-1}\mu_j$, for $j=2,3,4$. 
Suppose for some finite time interval we know both $P$ and $\tP$ are Hilbert--Schmidt operators whose
kernels also lie in $\mathrm{Dom}(\pa^4)$, and they are smooth in time. Here $\mathrm{Dom}(\pa^4)$ denotes
the subset of kernels which are four times continuously differentiable with respect to $x$,
with all square-integrable. 
Note by the ideal property for Hilbert--Schmidt operators, the operators $Q$ and $\tQ$ are trace-class
and smooth in time on the finite time interval.
Assume the Fredholm determinants $\mathrm{det}(\id+Q(0))$ and $\mathrm{det}(\id+\tQ(0))$ are non-zero;
see the end of this section for more details on Fredholm determinants.
Then there exists a possibly shorter finite time interval on which the Fredholm determinant associated
with $Q=Q(t)$ and $\tQ=\tQ(t)$ are non-zero. Further, on that time interval, there exist unique
solutions $G$ and $\tG$ to the linear Fredholm equations $P=G(\id+Q)$ and $\tP=\tG(\id+\tQ)$,
respectively, whose kernels lie in $\mathrm{Dom}(\pa^4)$ and are smooth in time.
These conclusions are established in Doikou \textit{et al.}~\cite{DMSW20}. 

Now suppose the complex matrix-valued functions $p=p(x,t)$ and $\tilde{p}=\tilde{p}(x,t)$
satisfy the respective linear equations,
\begin{equation*}
\pa_tp=d(\pa)p\quad\text{and}\quad\pa_t\tilde{p}=\tilde{d}(\pa)\tilde{p}, 
\end{equation*}
where
\begin{equation*}
d(\pa)\coloneqq\mu_2\pa^2+\mu_3\pa^3+\mu_4\pa^4
\quad\text{and}\quad
\tilde{d}(\pa)\coloneqq\tilde{\mu}_2\pa^2+\tilde{\mu}_3\pa^3+\tilde{\mu}_4\pa^4,
\end{equation*}
with $p(x,0)=p_0(x)$ and $\tilde{p}(x,0)=\tilde{p}_0(x)$ for all $x\in\R$ for given
complex matrix valued functions $p_0$ and $\tilde{p}_0$.
For $w:\R\to\mathbb{R}_+$, let $L^2_w$ denote the space of complex matrix-valued functions $f$
on $\R$ whose $L^2$-norm weighted by $w$ is finite, i.e.\/
\begin{equation*}
  \|f\|_{L^2_w}\coloneqq\int_{\R} \mathrm{tr}\,\bigl(f^\dag(x)f(x)\bigr)w(x)\,\mathrm{d}x<\infty,
\end{equation*} 
where $f^\dag$ denotes the complex-conjugate transpose of $f$ and `$\mathrm{tr}$' is the trace operator. 
Let $W:\mathbb{R}\to\mathbb{R}_+$ denote the function $W:x\mapsto 1+x^2$.
Further let $H$ denote the Sobolev space of complex matrix-valued functions who themselves,
as well as derivatives $\pa$ to all orders of them, are square-integrable.
\begin{definition}[Dispersion property]\label{def:dispersionproperty}
We say the constant coefficient polynomial operator $d=d(\pa)$ 
satisfies the \emph{dispersion property} if for all $\kappa\in\R$:
\begin{equation*}
\bigl(d(\mathrm{i}\kappa)\bigr)^\ast=-d(\mathrm{i}\kappa).
\end{equation*}
\end{definition}
Suppose $d(\pa)=\mu_2\pa^2+\mu_3\pa^3+\mu_4\pa^4$ satisfies the dispersion property.
This places a restriction on the parameters $\mu_2,\mu_3,\mu_4\in\CC$, in particular that
$\mu_2$ and $\mu_4$ are pure imaginary parameters and $\mu_3$ is a real parameter.
Then Doikou \textit{et al.\/} \cite[Lemma~3.1]{DMS20} establish if 
$p_0\in H\cap L^2_W$ then $p\in C^\infty\bigl([0,\infty);H\cap L^2_W\bigr)$ 
and the corresponding Hankel operator $P=P(t)$ is Hilbert--Schmidt valued.
Analogous results carry over to the operator $\tP$ under the assumption
$\tilde{d}=\tilde{d}(\pa)$ satisfies the dispersion property, 
with the corresponding restrictions imposed on 
$\tilde{\mu}_2,\tilde{\mu}_3,\tilde{\mu}_4\in\CC$.

\begin{remark}[Well-posedness of the linear system]\label{remark:well-posedness}
The well-posedness of solutions to the linear operator system in Definition~\ref{def:linearoperatorsystem},
under the dispersion property assumption for $d=d(\pa)$ and $\tilde{d}=\tilde{d}(\pa)$,
can be established by combining the conclusions of the previous two paragraphs as follows.
More details can be found in Doikou \textit{at al.}~\cite{DMSW20} and Doikou \textit{at al.}~\cite{DMS20}.
Assume $d=d(\pa)$ and $\tilde{d}=\tilde{d}(\pa)$ satisfy the dispersion property
and $p_0\in H\cap L^2_W$ and $\tilde{p}_0\in H\cap L^2_W$.
Hence Hankel operators $P_0$ and $\tP_0$ of the form given in Definition~\ref{def:Hankel},
respectively generated from $p_0$ and $\tilde{p}_0$, are Hilbert--Schmidt operators.
Naturally we assume the matricies $p_0$ and $\tilde{p}_0$
are commensurate so the matrix products $\tilde{p}_0p_0$ and $p_0\tilde{p}_0$ make sense.
We further assume the trace-class operators $Q_0\coloneqq\tP_0 P_0$ and $\tQ_0\coloneqq P_0\tP_0$ 
are such that $\mathrm{det}(\id+Q_0)\neq0$ and $\mathrm{det}(\id+\tQ_0)\neq0$. 
Then we deduce the commensurate solutions $p=p(y+x,t)$ and $\tilde{p}=\tilde{p}(y+x,t)$
to the respective linear partial differential equations $\pa_tp=d(\pa)p$
and $\pa_t\tilde{p}=\tilde{d}(\pa)\tilde{p}$ are such that 
$p\in C^\infty\bigl([0,\infty);H\cap L^2_W\bigr)$ and
$\tilde{p}\in C^\infty\bigl([0,\infty);H\cap L^2_W\bigr)$ with
$p(x,0)=p_0(x)$ and $\tilde{p}(x,0)=\tilde{p}_0(x)$ for all $x\in\R$.
The corresponding respective Hankel operators $P=P(x,t)$ and $\tP=\tP(x,t)$
are Hilbert--Schmidt operators and smooth functions of $x\in\R$ and $t\in[0,\infty)$.
The kernel function $q=q(y,z;x,t)$ corresponding to $Q=Q(x,t)$ given by
\begin{equation*}
q(y,z;x,t)=\int_{\R_-} \tilde{p}(y+\xi+x,t)p(\xi+z+x,t)\,\mathrm{d}\xi,
\end{equation*}
generates a trace-class operator and is a smooth function of $x$ and $t$, where $\R_-\coloneqq(-\infty,0]$.
The kernel function $\tilde q=\tilde q(y,z;x,t)$ corresponding to $\tQ=\tQ(x,t)$ is
defined similarly, with the positions of $p$ and $\tilde p$ exchanged, and possesses
the same properties. Further there exists a $T>0$ such that for $t\in[0,T]$ we know:
$\det(\id+Q(x,t))\neq0$ and $\det(\id+\tQ(x,t))\neq0$ and there exists a unique smooth function $g$ 
satisfying the linear Fredholm equation given by,
\begin{equation*}
p(y+z+x;t)=g(y,z;x,t)+\int_{\R_-} g(y,\xi;x,t)q(\xi,z;x,t)\,\mathrm{d}\xi,
\end{equation*}
as well as a unique smooth $\tilde{g}$ satisfying an analogous linear Fredholm
equation but with $p$, $q$ and $g$ respectively replaced by $\tilde p$, $\tilde q$ and $\tilde g$.
\end{remark}

Finally we briefly outline why the solution flow for $G$ prescribed in
Definition~\ref{def:linearoperatorsystem}, or in fact for the inflated system
for $G$ and $\tilde G$ shown therein, is a Fredholm Grassmannian flow.
Details on Fredholm Grassmann manifolds can be found in Pressley and Segal~\cite{PS}
or more recently in Abbondandolo and Majer~\cite{AMajer} or Andruchow and Larotonda~\cite{AL}.   
Their connection to integrable systems can be found in Sato~\cite{SatoI,SatoII} and
Segal and Wilson~\cite{SW}. For more explicit details on the connection between
Fredholm Grassmannian flows and the flow prescribed by the linear system of
equations in Definition~\ref{def:linearoperatorsystem}, see 
Beck \textit{et al.\/ }~\cite{BDMSI,BDMSII},
Doikou \textit{et al.\/ }~\cite[Sec.~2.3]{DMS20}
and Doikou \textit{et al.\/ }~\cite{DMSW20}. 
Suppose $\Hb$ is a given separable Hilbert space.
The Fredholm Grassmannian of all subspaces of $\Hb$ that are comparable
in size to a given closed subspace $\Vb\subset\Hb$ is defined as follows;  
see for example Segal and Wilson~\cite{SW}. 

\begin{definition}[Fredholm Grassmannian]\label{def:FredholmGrassmannian}
Let $\Hb$ be a separable Hilbert space with a given decomposition
$\Hb=\Vb\oplus\Vb^\perp$, where $\Vb$ and $\Vb^\perp$ are infinite
dimensional closed subspaces. The Grassmannian $\Gr(\Hb,\Vb)$
is the set of all subspaces $\Wb$ of $\Hb$ such that:
\begin{enumerate}
\item[(i)] The orthogonal projection $\mathrm{pr}\colon\Wb\to\Vb$ is 
a Fredholm operator, indeed it is a Hilbert--Schmidt perturbation
of the identity; and
\item[(ii)] The orthogonal projection $\mathrm{pr}\colon\Wb\to\Vb^\perp$ 
is a Hilbert--Schmidt operator.
\end{enumerate}
\end{definition}

Since $\Hb$ is separable, any element in $\Hb$ has a representation
on a countable basis, for example via the sequence of coefficients of the basis elements.
Suppose we are given a set of independent sequences in $\Hb=\Vb\oplus\Vb^\perp$
which span $\Vb$ and we record them as columns in the infinite matrix
\begin{equation*}
W=\begin{pmatrix} \id+Q \\ P \end{pmatrix}.
\end{equation*}
In other words, each column of $\id+Q\in\Vb$ and each column of $P\in\Vb^\perp$.
Assume also that when we constructed $\id+Q$ we ensured it was a Fredholm operator
on $\Vb$ with $Q\in\mathfrak J_2(\Vb;\Vb)$, where $\mathfrak J_2(\Vb;\Vb)$ is the
class of Hilbert--Schmidt operators from $\Vb$ to $\Vb$, equipped with
the norm $\|Q\|_{\mathfrak J_2}^2\coloneqq\mathrm{tr}\,Q^\dag Q$ where `$\mathrm{tr}$'
is the trace operator. The space $\mathfrak J_1(\Vb;\Vb)$ denotes the space
of trace-class operators, and for any $Q\in\mathfrak J_n$, $n=1,2$, we can define
\begin{equation*}
\mathrm{det}_n(\id+Q)\coloneqq\exp\Biggl(\sum_{k\geqslant n}\frac{(-1)^{k-1}}{k}\mathrm{tr}\,(Q^k)\Biggr).
\end{equation*}
This is the Fredholm determinant when $n=1$, and the regularised Fredholm determinant
when $n=2$, see Simon~\cite{Simon:Traces}. Naturally the operator $\id+Q$ is invertible
if and only if $\mathrm{det}_n(\id+Q)\neq0$. We also assume we constructed $P$
to ensure $P\in\mathfrak J_2(\Vb;\Vb^\perp)$, the space of Hilbert--Schmidt
operators from $\Vb$ to $\Vb^\perp$. Let $\Wb$ denote the subspace of $\Hb$
represented by the span of the columns of $W$. Let $\Vb_0\cong\Vb$ denote the canonical
subspace of $\Hb$ with the representation
\begin{equation*}
V_0=\begin{pmatrix} \id\\ O \end{pmatrix},
\end{equation*}
where $O$ is the infinite matrix of zeros. The projections 
$\mathrm{pr}\colon\Wb\to\Vb_0$ and $\mathrm{pr}\colon\Wb\to\Vb_0^\perp$
respectively generate
\begin{equation*}
W^\parallel=\begin{pmatrix} \id+Q \\ O \end{pmatrix}
\quad\text{and}\quad
W^\perp=\begin{pmatrix} O \\ P \end{pmatrix}.
\end{equation*}
This projection is achievable if and only if $\mathrm{det}_2(\id+Q)\neq0$.
Assume this is the case for the moment. We observe that the subspace
of $\Hb$ represented by the span of the columns of $W^\parallel$
coincides with the subspace $\Vb_0$, indeed, the transformation
$(\id+Q)^{-1}\in\mathrm{GL}(\Vb)$ transforms $W^\parallel$ to $V_0$.
Under the same transformation the representation $W$ for $\Wb$ becomes
\begin{equation*}
\begin{pmatrix} \id \\ G \end{pmatrix},
\end{equation*}
where $G=P(\id+Q)^{-1}$. We observe that any subspace $\Wb$ that can
be projected onto $\Vb_0$ can be represented in this way and vice-versa.
In this representation the operators $G\in\mathfrak J_2(\Vb,\Vb^\perp)$
parameterise all the subspaces $\Wb$ that can be projected on to $\Vb_0$.
If $\mathrm{det}_2(\id+Q)=0$ so the projection above is not possible,
we need to choose a different representative coordinate chart/patch.
This effectively corresponds to choosing a subspace $\Vb_0^\prime\cong\Vb$ of $\Hb$
different to canonical subspace $\Vb_0$ indicated above such that
the projection $\Wb\to\Vb_0^\prime$ is an isomorphism. This is
always possible; see Pressley and Segal~\cite[Prop.~7.1.6]{PS}.
For further details on coordinate patches see Beck \textit{et al.}~\cite{BDMSI}
and Doikou \textit{et al.}~\cite{DMSW20}, and for the analogous argument
for the inflated system shown in Definition~\ref{def:linearoperatorsystem}
see Doikou \textit{et al.}~\cite{DMS20}.
We also discuss the implications of $\id+Q$ becoming singular and
the necessity for different coordinate patches briefly in
the Discussion Section~\ref{sec:discussion}.

\section{Non-commutative fourth order quintic nonlinear Schr\"odinger}\label{sec:main}
We now assume $\tilde{\mu}_2=-\mu_2$, $\tilde{\mu}_3=\mu_3$ and $\tilde{\mu}_4=-\mu_4$.
For our main result below we also assume $d=d(\pa)$ and $\tilde{d}=\tilde{d}(\pa)$
satisfy the dispersion property in Definition~\ref{def:dispersionproperty}.
Hence we assume the parameters $\mu_2,\mu_3,\mu_4\in\CC$ are such that
$\mu_2,\mu_4\in\mathrm{i}\R$, the set of pure imaginary numbers, and $\mu_3\in\R$.
We use the notation $P^\dag$ to denote the operator adjoint
to the Hilbert--Schmidt operator $P$. In other words, if the Hilbert--Schmidt operator $P$
has the kernel $p$, then $P^\dag$ is the Hilbert--Schmidt operator whose kernel is the
complex-conjugate transpose of $p$ which we also denote by $p^\dag$. 
\begin{theorem}[Quintic kernel nonlinear Schr\"odinger equation]\label{thm:main}
Assume $P$, $\tP$, $Q$, $\tQ$, $G$ and $\tG$
satisfy the linear operator system in Definition~\ref{def:linearoperatorsystem} and
all the assumptions outlined in Remark~\ref{remark:well-posedness}.
Then for some $T>0$, the integral kernel $[G]=[G](y,z;x,t)$ satisfies,
for every $t\in[0,T]$, the matrix kernel equation:
\begin{align*}
(\pa_t-\mu_2\pa^2-&\mu_3\pa^3-\mu_4\pa^4)[G]\\
=&\;2\mu_2[G][\tG][G]+3\mu_3\Bigl(\bigl(\pa[G]\bigr)[\tG][G]+[G][\tG]\bigl(\pa[G]\bigr)\Bigr)\\
&\;+2\mu_4\Bigl(2\bigl(\pa^2[G]\bigr)[\tG][G]+[G]\bigl(\pa^2[\tG]\bigr)[G]+2[G][\tG]\bigl(\pa^2[G]\bigr)\\
&\;+\bigl(\pa[G]\bigr)\bigl(\pa[\tG]\bigr)[G]+3\bigl(\pa[G]\bigr)[\tG]\bigl(\pa[G]\bigr)
+[G]\bigl(\pa[\tG]\bigr)\bigl(\pa[G]\bigr)\\
&\;+3[G][\tG][G][\tG][G]\Bigr).
\end{align*}
As a special case, $[G](0,0;x,t)=g(0,0;x,t)$ satisfies the matrix equation:
\begin{align*}
(\pa_t-\mu_2\pa^2-\mu_3\pa^3-\mu_4\pa^4)g
=&\;2\mu_2g\tg g+3\mu_3\bigl((\pa g)\tg g+g\tg(\pa g)\bigr)\\
&\;+2\mu_4\bigl(2(\pa^2 g)\tg g+g(\pa^2\tg)g+2g\tg(\pa^2g)\\
&\;+(\pa g)(\pa\tg)g+3(\pa g)\tg(\pa g)+g(\pa\tg)(\pa g)\\
&\;+3g\tg g\tg g\bigr).
\end{align*}
In particular, a consistent choice for $\tP$ is $\tP=P^\dag$, whence $\tG=G^\dag$ and $\tilde{g}=g^\dag$. 
\end{theorem}
\begin{proof}
Recall $PU=VP$ and $\tP V=U\tP$ from Corollary~\ref{cor:invId}.
We split the proof into the following steps.\smallskip

\emph{Step~1: Applying the linear dispersion operator to $G=PU$.}  
With $G=PU$, using the Leibniz rule, that $P$ satisfies the linear
operator equation in Definition~\ref{def:linearoperatorsystem}
and also the identities for $\pa U$, $\pa^2 U$ and so forth from
Lemma~\ref{lemma:invOp}, we find:
\begin{align*}
\pa_t G-\mu_2&\pa^2G-\mu_3\pa^3G-\mu_4\pa^4G\\
=&\;(\pa_tP)U-PU(\pa Q)U-\mu_2\bigl((\pa^2P)U+2(\pa P)(\pa U)+P(\pa^2U)\bigr)\\
&\;-\mu_3\bigl((\pa^3P)U+3(\pa^2 P)(\pa U)+3(\pa P)(\pa^2 U)+P(\pa^3U)\bigr)\\
&\;-\mu_4\bigl((\pa^4P)U+4(\pa^3 P)(\pa U)+6(\pa^2P)(\pa^2U)+4(\pa P)(\pa^3 U)+P(\pa^4U)\bigr)\\
=&\;-PU(\pa_tQ-\mu_2\pa^2Q-\mu_3\pa^3Q-\mu_4\pa^4Q)U\\
&\;+2\mu_2\bigl((\pa P)U(\pa Q)P+P(\pa U)(\pa Q)U\bigr)\\
&\;+3\mu_3\bigl((\pa^2P)U(\pa Q)U+2(\pa P)(\pa U)(\pa Q)U+(\pa P)U(\pa^2 Q)U\\
&\;\qquad\quad+P(\pa^2U)(\pa Q)U+P(\pa U)(\pa^2 Q)U\bigr)\\
&\;+2\mu_4\bigl(2(\pa^3P)U(\pa Q)U+6(\pa^2P)(\pa U)(\pa Q)U+3(\pa^2P)U(\pa^2Q)U\\
&\;\qquad\quad+6(\pa P)(\pa^2U)(\pa Q)U+6(\pa P)(\pa U)(\pa^2Q)U+2(\pa P)U(\pa^3Q)U\\
&\;\qquad\quad+2P(\pa^3U)(\pa Q)U+3P(\pa^2U)(\pa^2Q)U+2P(\pa U)(\pa^3Q)U\bigr).
\end{align*}
For the moment we focus on the first term on the right just above.
Using $\tilde{\mu}_2=-\mu_2$, $\tilde{\mu}_3=\mu_3$ and $\tilde{\mu}_4=-\mu_4$, we observe:
\begin{align*}
\pa_tQ-\mu_2&\pa^2Q-\mu_3\pa^3Q-\mu_4\pa^4Q\\
=&\;(\pa_t\tP)P+\tP(\pa_tP)-\mu_2\bigl((\pa^2\tP)P+2(\pa\tP)(\pa P)+\tP(\pa^2P)\bigr)\\
&\;-\mu_3\bigl((\pa^3\tP)P+3(\pa^2\tP)(\pa P)+3(\pa\tP)(\pa^2P)+\tP(\pa^3P)\bigr)\\
&\;-\mu_4\bigl((\pa^4\tP)P+4(\pa^3\tP)(\pa P)+6(\pa^2\tP)(\pa^2P)+4(\pa\tP)(\pa^3 P)+\tP(\pa^4P)\bigr)\\
=&\;-2\mu_2\pa\bigl((\pa\tP)P\bigr)-3\mu_3\pa\bigl((\pa\tP)(\pa P)\bigr)\\
&\;-2\mu_4\pa\bigl((\pa^3\tP)P+(\pa^2\tP)(\pa P)+2(\pa\tP)(\pa^2P)\bigr).
\end{align*}
Our proof now proceeds as follows. We substitute this last expression into the
first term on the right in the previous equation and apply the kernel bracket operator,
treating the coefficients of $\mu_2$, $\mu_3$ and $\mu_4$ separately.\smallskip

\emph{Step~2: Terms involving $\mu_2$.}
Applying the kernel bracket operator to these terms on the right in Step~1, modulo $2\mu_2$ we have:
\begin{align*}
\bigl[PU\pa\bigl((\pa\tP) P\bigr)U+(\pa P)U(\pa Q)P+P&(\pa U)(\pa Q)U\bigr]\\
=&\;[PU(\pa\tP)+(\pa P)U\tP+P(\pa U)\tP][PU]\\
=&\;\bigl[\pa(PU\tP)\bigr][G]\\
=&\;[G][\tG][G],
\end{align*}
where we used the result (i) from the key identities Lemma~\ref{lemma:keyidentities}.
We have thus generated the term involving $\mu_2$ on the right stated in the Theorem.\smallskip

\emph{Step~3: Terms involving $\mu_3$.}
Applying the bracket operator to these terms on the right in Step~1, 
using $\pa^2Q=\pa\bigl((\pa\tP)P+\tP(\pa P)\bigr)$, then modulo $3\mu_3$, we observe:
\begin{align*}
  \bigl[PU\pa&\bigl((\pa\tP)(\pa P)\bigr)U+(\pa^2P)U(\pa Q)U
    +2(\pa P)(\pa U)(\pa Q)U\\
    &\;+(\pa P)U(\pa^2 Q)U
    +P(\pa^2U)(\pa Q)U+P(\pa U)(\pa^2 Q)U\bigr]\\
=&\;\bigl[(\pa^2P)U\tP+2(\pa P)(\pa U)\tP+(\pa P)U(\pa\tP)+P(\pa^2U)\tP+P(\pa U)(\pa\tP)\bigr][PU]\\
&\;+\bigl[PU(\pa\tP)+(\pa P)U\tP+P(\pa U)\tP\bigr]\bigl[(\pa P)U\bigr].
\end{align*}
The pre-factors of $\bigl[(\pa P)U\bigr]$ simplify to $\bigl[\pa(PU\tP)\bigr]=[G][\tG]$ using 
the key identities Lemma~\ref{lemma:keyidentities}. Note $\bigl[(\pa P)U\bigr]$ itself is given by
\begin{equation*}
  \bigl[(\pa P)U\bigr]=\bigl[\pa(PU)\bigr]-\bigl[P(\pa U)\bigr]
  =\pa[G]+\bigl[PU(\pa Q)U\bigr]=\pa[G]+[PU\tP][G].
\end{equation*}
The pre-factors of $[PU]$ from the first expression of this step simplify to
\begin{align*}
  \bigl[\pa^2(PU\tP)-(\pa P)U(\pa\tP)-P&(\pa U)(\pa\tP)-PU(\pa^2\tP)(\pa\tP)\bigr]\\
  &=\bigl[\pa^2(PU\tP)\bigr]-\bigl[\pa\bigl(PU(\pa\tP)\bigr)\bigr]\\
  &=\pa\bigl([G][\tG]\bigr)-[G]\bigl(\pa[\tG]\bigr)-[G][\tG][PU\tP]\\
  &=\bigl(\pa [G]\bigr)[\tG]-[G][\tG][PU\tP],
\end{align*}
using the key identities Lemma~\ref{lemma:keyidentities}. Combining these last two results
we see the terms on the right in the first expression of this step are:
\begin{equation*}
  [G][\tG]\bigl(\pa[G]+[PU\tP][G]\bigr)+\bigl((\pa [G])[\tG]-[G][\tG][PU\tP]\bigr)[G],
\end{equation*}
which simplify to the terms involving $3\mu_3$ on the right stated in the Theorem.\smallskip

\emph{Step~4: Terms involving $\mu_4$.}
In practice we split our computation for these terms into several successive steps.
We apply the kernel bracket operator to these terms on the right in Step~1 and
use $\pa^2Q=\pa\bigl((\pa\tP)P+\tP(\pa P)\bigr)$ as well as
$\pa^2Q=\pa\bigl((\pa^2\tP)P+2(\pa\tP)(\pa P)+\tP(\pa^2 P)\bigr)$.
Collating terms with respective post-factors
$\bigl[(\pa^2 P)U\bigr]$, $\bigl[(\pa P)U\bigr]$ and $[PU]$,
then modulo $2\mu_4$, we observe:
\begin{align*}
\bigl[&2(\pa^3P)U(\pa Q)U+6(\pa^2P)(\pa U)(\pa Q)U+3(\pa^2P)U(\pa^2Q)U\\
&+6(\pa P)(\pa^2U)(\pa Q)U+6(\pa P)(\pa U)(\pa^2Q)U+2(\pa P)U(\pa^3Q)U\\
&+2P(\pa^3U)(\pa Q)U+3P(\pa^2U)(\pa^2Q)U+2P(\pa U)(\pa^3Q)U\bigr]\\
&\;\qquad\quad=2\bigl[PU(\pa\tP)+P(\pa U)\tP+(\pa P)U\tP\bigr]\bigl[(\pa^2P)U\bigr]\\
&\;\qquad\qquad+\bigl[PU(\pa^2\tP)+4P(\pa U)(\pa\tP)+3P(\pa^2U)\tP+4(\pa P)U(\pa\tP)\\
&\;\qquad\qquad\qquad+6(\pa P)(\pa U)\tP+3(\pa^2P)U\tP \bigr]\bigl[(\pa P)U\bigr]\\
&\;\qquad\qquad+\bigl[PU(\pa^3\tP)+2P(\pa U)(\pa^2\tP)+3P(\pa^2U)(\pa\tP)+2P(\pa^3U)\tP\\
&\;\qquad\qquad\qquad+2(\pa P)U(\pa^2\tP)+6(\pa P)(\pa U)(\pa\tP)+6(\pa P)(\pa^2U)\tP\\
&\;\qquad\qquad\qquad+3(\pa^2P)U(\pa\tP)+6(\pa^2P)(\pa U)\tP+2(\pa^3P)U\tP\bigr][PU].
\end{align*}
Observe the pre-factor of the term $\bigl[(\pa^2 P)U\bigr]$ is $2\bigl[\pa(PU\tP)\bigr]=2[G][\tG]$ using 
the key identities Lemma~\ref{lemma:keyidentities}. Note the factor $\bigl[(\pa^2 P)U\bigr]$ itself, using
inverse operator Lemma~\ref{lemma:invOp}, is given by
\begin{align*}
\bigl[(\pa^2 P)U\bigr]=&\;\bigl[\pa^2(PU)\bigr]-2\bigl[(\pa P)(\pa U)\bigr]-\bigl[P(\pa^2 U)\bigr]\\
=&\;\pa^2[G]+2\bigl[(\pa P)U(\pa Q)U\bigr]+2\bigl[P(\pa U)(\pa Q)U\bigr]+\bigl[PU(\pa^2Q)U\bigr]\\
=&\;\pa^2[G]+2\bigl[(\pa P)U\tP][PU]+2\bigl[P(\pa U)P\bigr][PU]\\
&\;+\bigl[PU(\pa\tP)\bigr][PU]+\bigl[PU\tP\bigr]\bigl[(\pa P)U\bigr].
\end{align*}
Note the final term on the right just above, using the second relation from Step~3, is given by 
\begin{equation*}
\bigl[PU\tP\bigr]\bigl[(\pa P)U\bigr]
=\bigl[PU\tP\bigr]\pa[G]+\bigl[PU\tP\bigr]^2[G]
=\bigl[PU\tP\bigr]\pa[G]-\bigl[P(\pa U)\tP\bigr][G].
\end{equation*}
If we substitute this into the previous result we find,
using the key identities Lemma~\ref{lemma:keyidentities}: 
\begin{align*}
\bigl[(\pa^2 P)U\bigr]
=&\;\pa^2[G]+\bigl[\pa(PU\tP)\bigr][G]+\bigl[(\pa P)U\tP\bigr][G]+[PU\tP]\pa[G]\\
=&\;\pa^2[G]+[G][\tG][G]+\bigl[(\pa P)U\tP\bigr][G]+[PU\tP]\pa[G].
\end{align*}
Hence the terms with post-factor $\bigl[(\pa^2 P)U\bigr]$ become
\begin{equation*}
2[G][\tG]\bigl(\pa^2[G]+[G][\tG][G]+\bigl[(\pa P)U\tP\bigr][G]+[PU\tP]\pa[G]\bigr).
\end{equation*}\smallskip

\emph{Step~5: Terms involving $\mu_4$ with post-factor $\bigl[(\pa P)U\bigr]$.}
We consider the terms with post-factor $\bigl[(\pa P)U\bigr]$ from the first
relation in Step~4. Modulo $2\mu_4$ the terms concerned equal 
$\bigl[3\pa^2\big(PU\tP\bigr)-2PU(\pa^2\tP)-2P(\pa U)(\pa\tP)-2(\pa P)U(\pa\tP)\bigr]$
which simplify to:
\begin{align*}
3\pa\bigl([G][\tG]\bigr)-2\bigl[\pa\bigl(PU(\pa\tP)\bigr)\bigr]
&=3\pa\bigl([G][\tG]\bigr)-2\bigl([G]\pa[\tG]+[G][\tG][PU\tP]\bigr)\\
&=3\bigl(\pa[G]\bigr)[\tG]+[G]\bigl(\pa[\tG]\bigr)-2[G][\tG][PU\tP],
\end{align*}
where we used the key identities Lemma~\ref{lemma:keyidentities}. 
Since from the second relation from Step~3 we know
$\bigl[(\pa P)U\bigr]=\pa[G]+\bigl[PU\tP\bigr][G]$, 
the terms with post-factor $\bigl[(\pa P)U\bigr]$ from the first
relation in Step~4 become
\begin{equation*}
\Bigl(3\bigl(\pa[G]\bigr)[\tG]+[G]\bigl(\pa[\tG]\bigr)-2[G][\tG][PU\tP]\Bigr)
\bigl(\pa[G]+[PU\tP][G]\bigr).
\end{equation*}\smallskip

\emph{Step~6: Terms involving $\mu_4$ with post-factor $\pa[G]$.}
We deal with the terms with post-factor $[PU]=[G]$
from the first relation in Step~4 in Step~7 momentarily.
However from Steps~4 and 5, besides the single term with post-factor
$\pa^2[G]$ in the last expression in Step~4, all the other
terms will have post-factors $\pa[G]$ or $[G]$. 
If we collect the terms from the very final relations in both Steps~4 and 5,
then the terms with post-factor $\pa[G]$ simplify to
\begin{equation*}
  \Bigl(3\bigl(\pa[G]\bigr)[\tG]+[G]\bigl(\pa[\tG]\bigr)]\Bigr)\pa[G].
\end{equation*}\smallskip

\emph{Step~7: Terms involving $\mu_4$ with post-factor $[G]$.}
There are terms with post-factor $[G]$ from the final 
relations in both Steps~4 and 5, which we re-introduce presently in Step~8.
However momentarily we focus on the terms with post-factor $[G]$
from the first relation in Step~4. Modulo $2\mu_4$ the terms concerned are equal to
\begin{align*}
\bigl[&2\pa^3(PU\tP)-PU(\pa^3\tP)-4P(\pa U)(\pa^2\tP)-3P(\pa^2U)(\pa\tP)\\
&\quad-4(\pa P)U(\pa^2\tP)-6(\pa P)(\pa U)(\pa\tP)-3(\pa^2P)U(\pa\tP)\bigr]\\
&=2\pa^2\bigl([G][\tG]\bigr)-\bigl[PU(\pa^3\tP)+4\bigl(\pa(PU)\bigr)(\pa^2\tP)
+3\bigl(\pa^2(PU)\bigr)(\pa\tP)\bigr]\\
&=2\pa^2\bigl([G][\tG]\bigr)-3\pa^2\bigl[\bigl(PU(\pa\tP)\bigr)\bigr]
+2\bigl[PU(\pa^3\tP)\bigr]+2\bigl[\bigl(\pa(PU)\bigr)(\pa^2\tP)\bigr]\\
&=2\pa^2\bigl([G][\tG]\bigr)-3\pa\bigl([G]\pa[\tG]\bigr)-3\bigl([G][\tG]\bigr)^2
-3\pa\bigl([G][\tG]\bigr)[PU\tP]+2\pa\bigl[PU(\pa^2\tP)\bigr],
\end{align*}
where we used relation (v) in the key identities Lemma~\ref{lemma:keyidentities} in the final step as
well as combined the final two terms on the right.
\smallskip

\emph{Step~8: Combining all terms involving $\mu_4$.}
We now combine all the terms involving $\mu_4$ together.
These are all the terms on the right in the final relation in Step~7,
for which we need to include the post-factor $[G]$, 
the two terms from the final expression in Step~6 which have
the post-factor $\pa[G]$, all the terms with
post-factor $[G]$ from the final expression in Step~5,
and finally all the terms with post-factor $[G]$ as well as the single
term with post-factor $\pa^2[G]$ from the final expression in Step~4.
Modulo $2\mu_4$, using that $[PU\tP]^2=\bigl[PU\pa(\tP P)U\tP\bigr]=-[P(\pa U)\tP]$,
these terms combine to give:
\begin{gather*}
  2[G][\tG]\pa^2[G]-\bigl([G][\tG]\bigr)^2[G]+3\bigl(\pa[G]\bigr)[\tG]\pa[G]\\
  +[G]\bigl(\pa[\tG]\bigr)\pa[G]+2\pa^2\bigl([G][\tG]\bigr)[G]-3\pa\bigl([G]\pa[\tG]\bigr)[G]\\
  +2\Bigl([G][\tG]\bigl[\bigl(\pa(PU)\bigr)\tP]-[G]\bigl(\pa[\tG]\bigr)[PU\tP]
  +\pa\bigl[PU(\pa^2\tP)\bigr]\Bigr)[G].
\end{gather*}
Substituting the result (iii) from the key identities Lemma~\ref{lemma:keyidentities}
and combining like terms then gives the first statement of the theorem.\smallskip

\emph{Step~9: Remaining statements.}
The second statement is a special case that follows by setting $y=z=0$ in the first statement.
For the third statement, if we suppose $\tP=P^\dag$, then $U=(\id+P^\dag P)^{-1}=U^\dag$ and
$\tG=P^\dag(\id+PP^\dag)^{-1}=(\id+P^\dag P)^{-1}P^\dag=UP^\dag=(PU)^\dag=G^\dag$.
\end{proof}
Judicious choices for $\tP$ generate reverse space-time and reverse time nonlocal versions
of the quintic nonlinear Schr\"odinger equations stated in Theorem~\ref{thm:main} as follows.
\begin{corollary}[Reverse space-time nonlocal matrix quintic nonlinear Schr\"odinger equation]\label{cor:reveresespacetime}
Suppose $\mu_3=0$ and recall we assume $\mu_2$ and $\mu_4$ are pure imaginary parameters.
If we choose $\tP(x,t)=P^{\mathrm{T}}(-x,-t)$, where $P^{\mathrm{T}}$ is the operator
whose matrix kernel is the transpose of the matrix kernel corresponding to $P$,
then the matrix kernel $[G]=[G](y,z;x,t)$ satisfies, for every $t\in[0,T]$,
the reverse space-time nonlocal matrix quintic nonlinear Schr\"odinger kernel equation given by the
equation for $[G]$ stated in Theorem~\ref{thm:main} with $\tG(x,t)=G^{\mathrm{T}}(-x,-t)$.
Further, setting $y=z=0$ generates the reverse space-time nonlocal matrix quintic nonlinear Schr\"odinger equation
for $g=g(0,0;x,t)$ stated in Theorem~\ref{thm:main} with
$\tilde{g}(0,0;x,t)=g^{\mathrm{T}}(0,0;-x,-t)$.
Similarly if we choose $\tP(x,t)=P^{\mathrm{T}}(x,-t)$ then we generate the
corresponding reverse time only nonlocal equations.
\end{corollary}
Before proving this corollary we give an example of the
reverse space-time nonlocal matrix quintic nonlinear Schr\"odinger equation
whose integrability is established by the corollary.
\begin{example}\label{ex:reversespace-time}
Suppose the complex matrix-valued function $p=p(x,t)$ satisfies
the linear partial partial differential equation $\pa_tp=\mu_2\pa^2p+\mu_4\pa^4p$
where $\mu_2$ and $\mu_4$ are pure imaginary parameters.
Then $\tilde{p}(x,t)=p^{\mathrm{T}}(-x,-t)$ satisfies the linear partial differential equation
$\pa_t\tilde p=-\mu_2\pa^2\tilde p-\mu_4\pa^4\tilde p$, consistent with the assumptions that
$\tilde{\mu_2}=-\mu_2$ and $\tilde{\mu_4}=-\mu_4$ at the very beginning of this section.
Suppose $P$ and $\tP$ are Hilbert--Schmidt Hankel operators with kernels $p$ and $\tilde{p}$, respectively.
Then Corollary~\ref{cor:reveresespacetime} establishes that if the operator $G$
satisfies $P=G(\id+Q)$ with $Q=\tP P$, then 
$[G](0,0;x,t)=g(0,0;x,t)$, with $\tg(0,0;x,t)=g^{\mathrm{T}}(0,0;-x,-t)$, satisfies the matrix equation:
\begin{align*}
(\pa_t-\mu_2\pa^2-\mu_4\pa^4)g
=&\;2\mu_2g\tg g+2\mu_4\bigl(2(\pa^2 g)\tg g+g(\pa^2\tg)g+2g\tg(\pa^2g)\\
&\;+(\pa g)(\pa\tg)g+3(\pa g)\tg(\pa g)+g(\pa\tg)(\pa g)+3g\tg g\tg g\bigr).
\end{align*}
\end{example}
\begin{proof} (\emph{of Corollary~\ref{cor:reveresespacetime}})
Recall since $\tilde{\mu}_2=-\mu_2$ and $\tilde{\mu}_4=-\mu_4$ the operators $P$ and $\tP$
satisfy the respective linear PDEs $\pa_tP=\mu_2\pa^2P+\mu_4\pa^4P$ and
$\pa_t\tP=-\mu_2\pa^2\tP-\mu_4\pa^4\tP$. The choice $\tP(x,t)=P^{\mathrm{T}}(-x,-t)$ is consistent
with these two equations. Recall $G=PU$ and $\tG=\tP V$ where $U=(\id+\tP P)^{-1}$ and $V=(\id+P\tP)^{-1}$.
If we substitute $\tP(x,t)=P^{\mathrm{T}}(-x,-t)$ into these expressions for $G$ and $\tG$ we observe
\begin{align*}
G^{\mathrm{T}}(-x,-t)&=\bigl(\id+P^{\mathrm{T}}(-x,-t)P(x,t)\bigr)^{-1}P^{\mathrm{T}}(-x,-t),
\intertext{while}
\tG(x,t)&=P^{\mathrm{T}}(-x,-t)\bigl(\id+P(x,t)P^{\mathrm{T}}(-x,-t)\bigr)^{-1}.
\end{align*}
We deduce $\tG(x,t)=G^{\mathrm{T}}(-x,-t)$ and the reverse space-time results follow.
The reverse time only result follows immediately.
\end{proof}
\begin{remark}[Sign of the nonlinear terms]
Suppose in the final statement of Theorem~\ref{thm:main} we instead
make the choice $\tP=-P^\dag$. This choice is still consistent with
the linear partial differential equations satisfied by $P$ and $\tP$
in the linear operator system in Definition~\ref{def:linearoperatorsystem},
with $\tilde{\mu}_2=-\mu_2$, $\tilde{\mu}_3=\mu_3$ and $\tilde{\mu}_4=-\mu_4$
and $\mu_2$ and $\mu_4$ chosen pure imaginary while $\mu_3$ chosen to be real.
With this choice for $\tP$ we observe $V=(\id-PP^\dag)^{-1}=V^\dag$ and thus
$\tG=\tP V=-P^\dag V=-(VP)^\dag=-G^\dag$. Hence with the choice $\tP=-P^\dag$,
we generate the equations for $[G]$ and $g$ shown in Theorem~\ref{thm:main}
but with $\tG=-G^\dag$ instead. This has the effect of changing the sign of
all the degree three terms while the sign of the quintic term is unchanged.
\end{remark}
\begin{remark}[Parameter restrictions]\label{rmk:pracissues}
For Theorem~\ref{thm:main} we assumed $\mu_2,\mu_4\in\mathrm{i}\R$ and $\mu_3\in\R$
to ensure $d=d(\pa)$ satisfied the dispersion property in Definition~\ref{def:dispersionproperty}.
This ensures suitable regularity for the kernel function $p=p(x,t)$ which solves
the underlying linear partial differential equation. That in turn ensures a suitable solution
to the linear Fredholm equation for $G$; recall the discussion in Remark~\ref{remark:well-posedness}.
If the parameters $\mu_2$, $\mu_3$ and $\mu_4$ are more general so
the dispersion property does not hold,
then establishing suitable regularity for $p=p(x,t)$ requires further investigation.
With regards the choices of the parameters $\tilde{\mu}_j$ we made, for $j=2,3,4$,
we in principle could choose these differently to $\tilde{\mu}_j=(-1)^{j-1}\mu_j$. 
For example, the choice $\tilde{\mu}_4=-2\mu_4$ appears to be consistent and may
lead to different equations.
\end{remark}

\section{Numerical simulations}\label{sec:numericalsimulations}
We present numerical simulations for the commutative version
of the fourth order quintic nonlinear Schr\"odinger equation
from Theorem~\ref{thm:main}. The commutative version of these
equations appear most frequently in practice in the modelling of
ultrashort pulse propagation in optical fibres; see for example
Kang \textit{et al.} \cite{KangXiaMa}, Guo \textit{et al.} \cite{GuoHaoGu} 
or Wang \textit{et al.} \cite{WangPorsezianHe}.
We chose the parameter values $\mu_2=-\mathrm{i}$, $\mu_3=1$ and $\mu_4=\mathrm{i}$
for the purposes of simulation. We provide two independent 
simulations to generate approximations to the solution
of the commutative version of the equation for $g=g(0,0;x,t)$ in Theorem~\ref{thm:main}
as follows: (i) Direct numerical simulation based on an adaptation of a
well-known spectral algorithm. This advances the approximate solution in Fourier space,
denoted by $u_m$, in successive time steps; 
(ii) Generation of an approximation to the solution, denoted by $\hat g$,
by using the Grassmann--P\"oppe method, i.e.\/ based on the analytical
linearisation approach we have outlined in Sections~\ref{sec:preliminaries} and \ref{sec:main}.
For both numerical methods we chose an initial profile function $p_0=p_0(x)$
on the interval $[-L/2,L/2]$ for some fixed domain length $L>0$.
Recall from Definition~\ref{def:linearoperatorsystem}, assuming we set $\tilde P=P^\dag$, then given 
the Hankel operator $P=P(x,t)$ or equivalently its kernel $p=p(y+z+x;t)$,
the functions $q=q(y,z;x,t)$ and $g=g(y,z;x,t)$ are prescribed by the equations
\begin{align*}
q(y,z;x,t)&=\int_{-L/2}^0p^\dag(y+\zeta+x;t)p(\zeta+z+x;t)\,\rd\zeta
\intertext{and}
p(y+z+x;t)&=g(y,z;x,t)+\int_{-L/2}^0g(y,\zeta;x,t)q(\zeta,z;x,t\,\rd\zeta.
\end{align*}
Note in this commutative setting with $p$ scalar, the quantity $p^\dag=p^\ast$
in the definition for $q$ above is just the complex conjugate of $p$.
In our simulations we chose the initial profile $p_0(x)=0.15\,\mathrm{sech}(x/40)$ with $L=40$.

First, let us outline the direct numerical simulation method.
The first task in this method is to compute the initial data $g_0$.
To achieve this we compute the initial auxiliary data function $q_0=q_0(y,z;x)$
and then initial data profile $g_0=g_0(y,z;x)$ by numerically solving the equations
just above with $p$, $q$ and $g$ replaced respectively by $p_0$, $q_0$ and $g_0$.
Note $q_0$ is just defined in terms of the data function $p_0$ and the
integral involved is approximated using a left-hand Riemann sum.
On the other hand we must solve the second equation above for $g_0$.
We achieve this by approximating the integral again by a left-hand Riemann sum
and solving the resulting large linear algebraic system of equations.
Hence, we generated the approximation $\hat g_0=\hat g_0(y,z;x)$
from the profile for $p_0$ stated using the procedure just outlined.
We set the initial data for the direct numerical method $u_0$ to be the
Fourier transform of $\hat g_0(0,0;x)$. 

The Fourier spectral method we used to advance the commutative fourth order 
quintic nonlinear Schr\"odinger solution forward in time is as follows.
Suppose $\Psi=\Psi(g,\pa g,\pa^2 g)$ denotes the quintic nonlinear function
of $g$, $\pa g$ and $\pa^2 g$ with $\tilde g=g^\ast$ shown on the right-hand side
in the evolution equation for $g$ given in Theorem~\ref{thm:main}.
Naturally in the current commutative setting, the functional form for $\Psi$
shown can be further simplified. The Fourier spectral split-step method we use is given by:
\begin{align*}
v_m&=\exp\bigl(\Delta t(\mu_2K^2+\mu_3K^3+\mu_4K^4)\bigr)u_m,\\
u_{m+1}&=v_m
+\Delta t\,\mathcal F\Bigl(\Psi\bigl(\mathcal F^{-1}(v_m),\mathcal F^{-1}(Kv_m),\mathcal F^{-1}(K^2v_m)\bigr)\Bigr),
\end{align*}
where $\mathcal F$ denotes the Fourier transform and $K$
is the diagonal matrix of Fourier coefficients $2\pi\mathrm{i}k$.
In practice we use the fast Fourier transform and its inverse. 
We chose $\Delta t=0.001$ and the number of Fourier modes is $2^8$.
The result is shown in the top two panels in Figure~\ref{fig:qNLS}.

Second, in the Grassmann--P\"oppe method, given the initial profile 
$p_0=p_0(x)$ we analytically advance the Fourier transform $\mathcal F(p_0)$ of the initial
data in Fourier space to any time of interest $t\in[0,T]$, where $T=100$.
To generate $p=p(x,t)$ at that time, we then compute the inverse Fourier transform.
In other words we compute the approximation
$\hat p(x,t)=\mathcal F^{-1}(\exp(t(\mu_2K^2+\mu_3K^3+\mu_4K^4))\mathcal F(p_0))$
for any given time $t$. We then compute an approximation $\hat q=\hat q(y,z;x,t)$
at the time $t$ by approximating the integral in the prescription for $q$ above
by a left-hand Riemann sum. To generate an approximation $\hat g$ for $g=g(y,z;x,t)$
at that time we approximate the integral in the linear Fredholm integral equation prescribing 
$g$ above by a left-hand Riemann sum and solve the resulting linear algebraic system
of equations for $\hat g=\hat g(y,z;x,t)$. An approximation to the solution
of the commutative fourth order quintic nonlinear Schr\"odinger equation is
then $\hat g(0,0;x,t)$. The result is shown in the middle two panels in Figure~\ref{fig:qNLS}.
The magnitude of the difference between $\hat g(0,0;x,t)$ and $u_m$ 
is shown in the lower left panel. The lower right panel shows the magnitude of 
the Fredholm determinant $\det(\id+\hat Q(x,t))$, where $\hat Q=\hat Q(x,t)$
is the linear operator associated with the kernel function $\hat q=\hat q(y,z;x,t)$.

\begin{figure}
  \begin{center}
  \includegraphics[width=6cm,height=5cm]{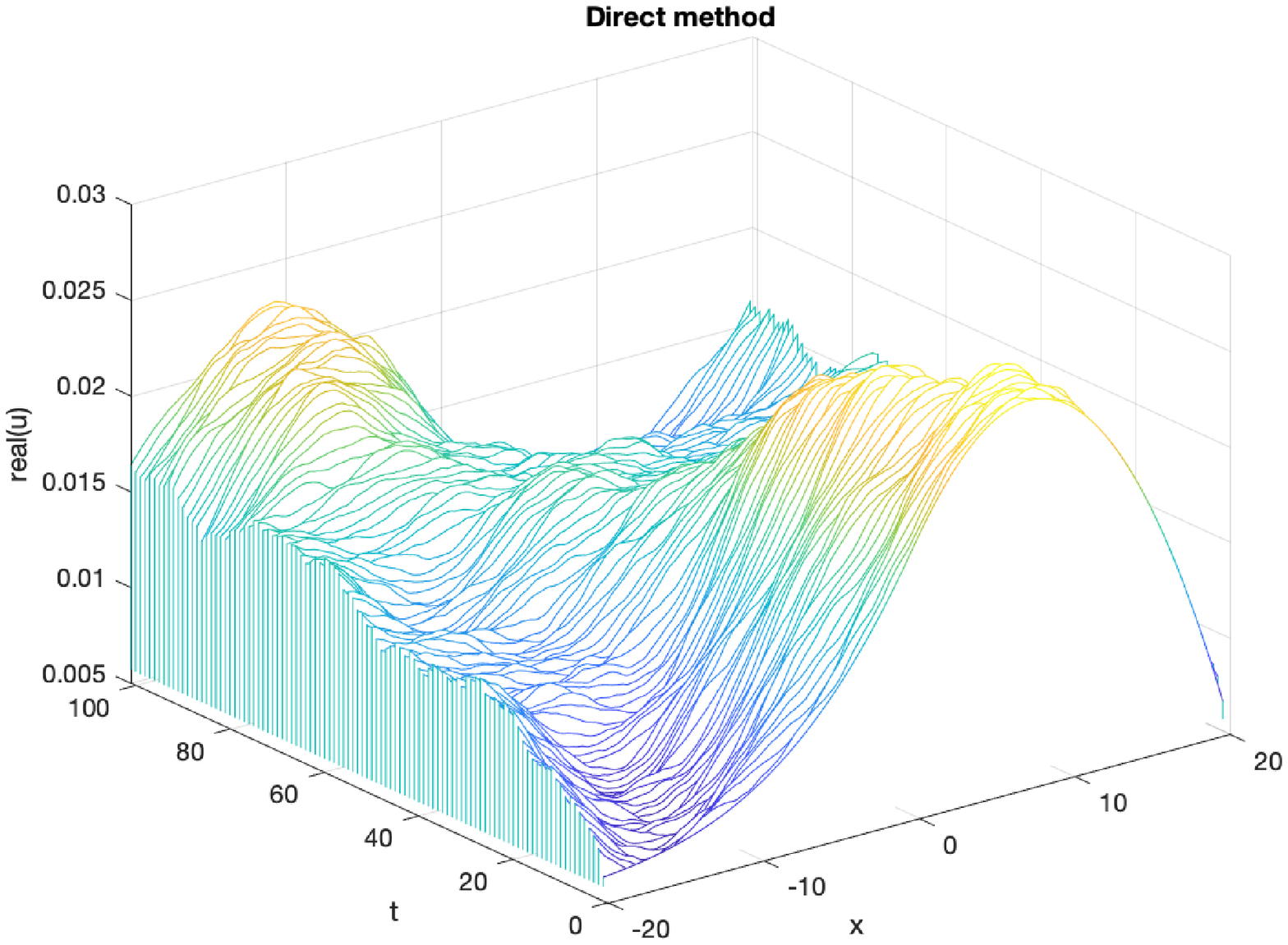}
  \includegraphics[width=6cm,height=5cm]{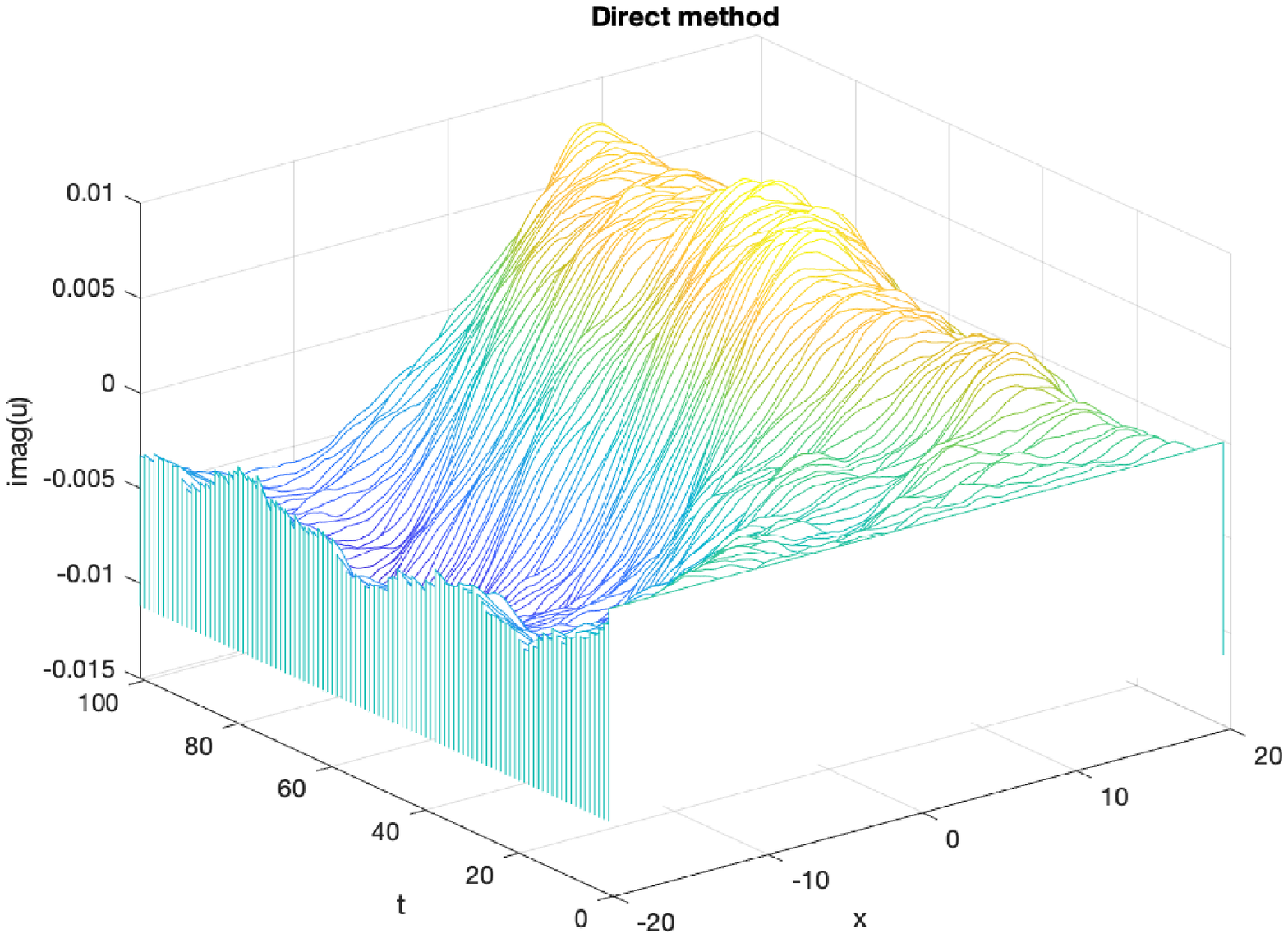}\\
  \includegraphics[width=6cm,height=5cm]{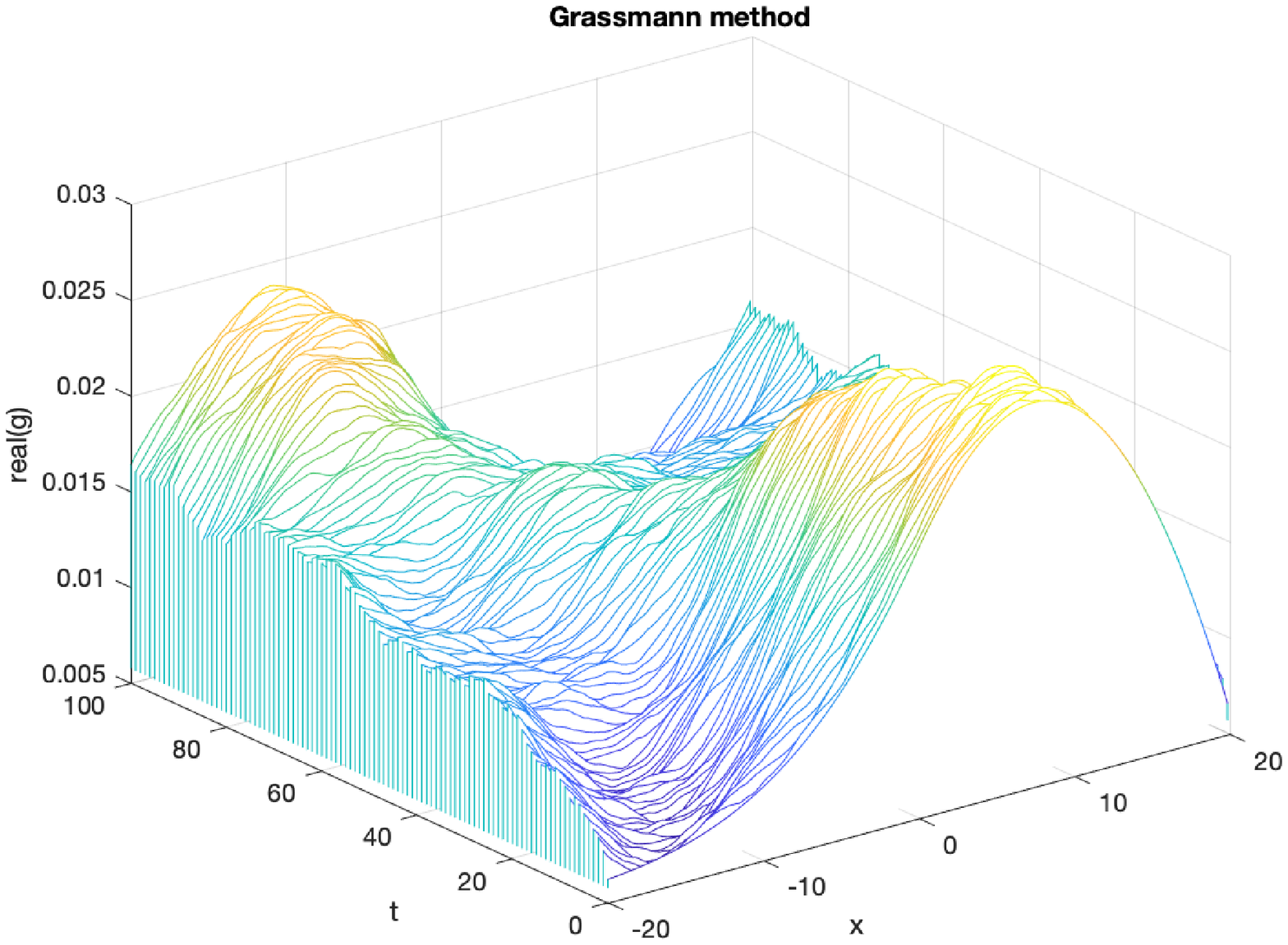}
  \includegraphics[width=6cm,height=5cm]{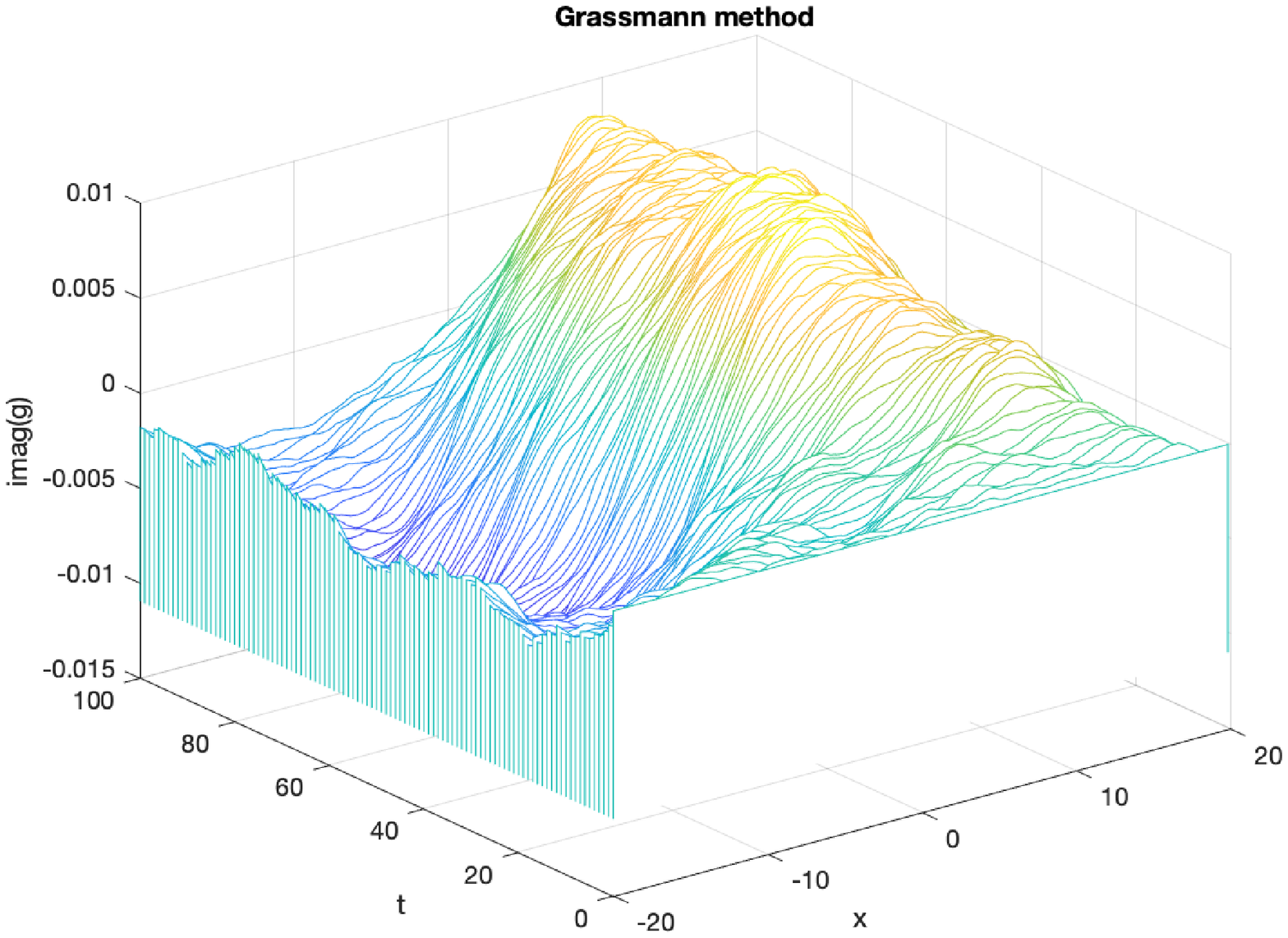}\\
  \includegraphics[width=6cm,height=5cm]{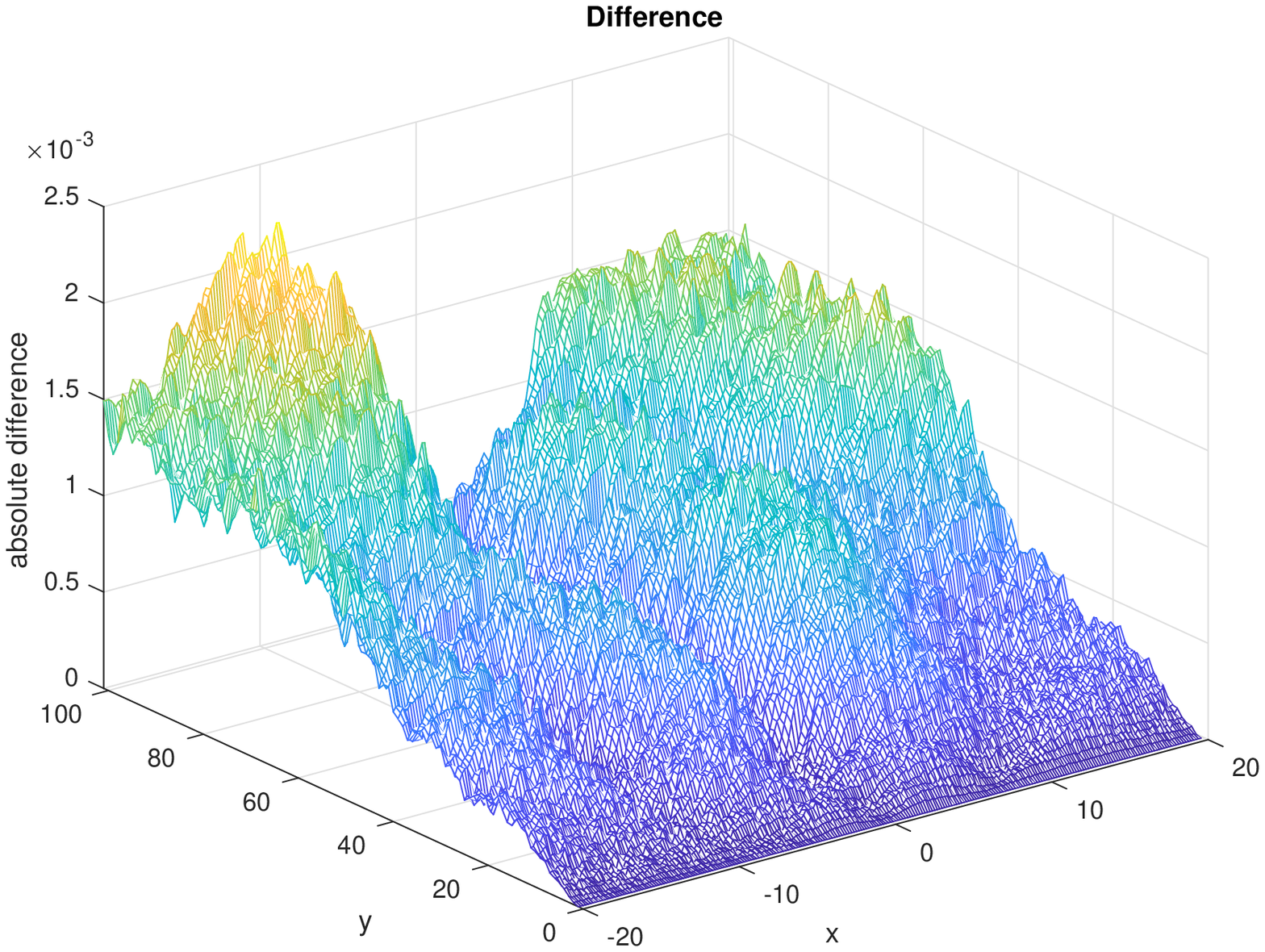}
  \includegraphics[width=6cm,height=5cm]{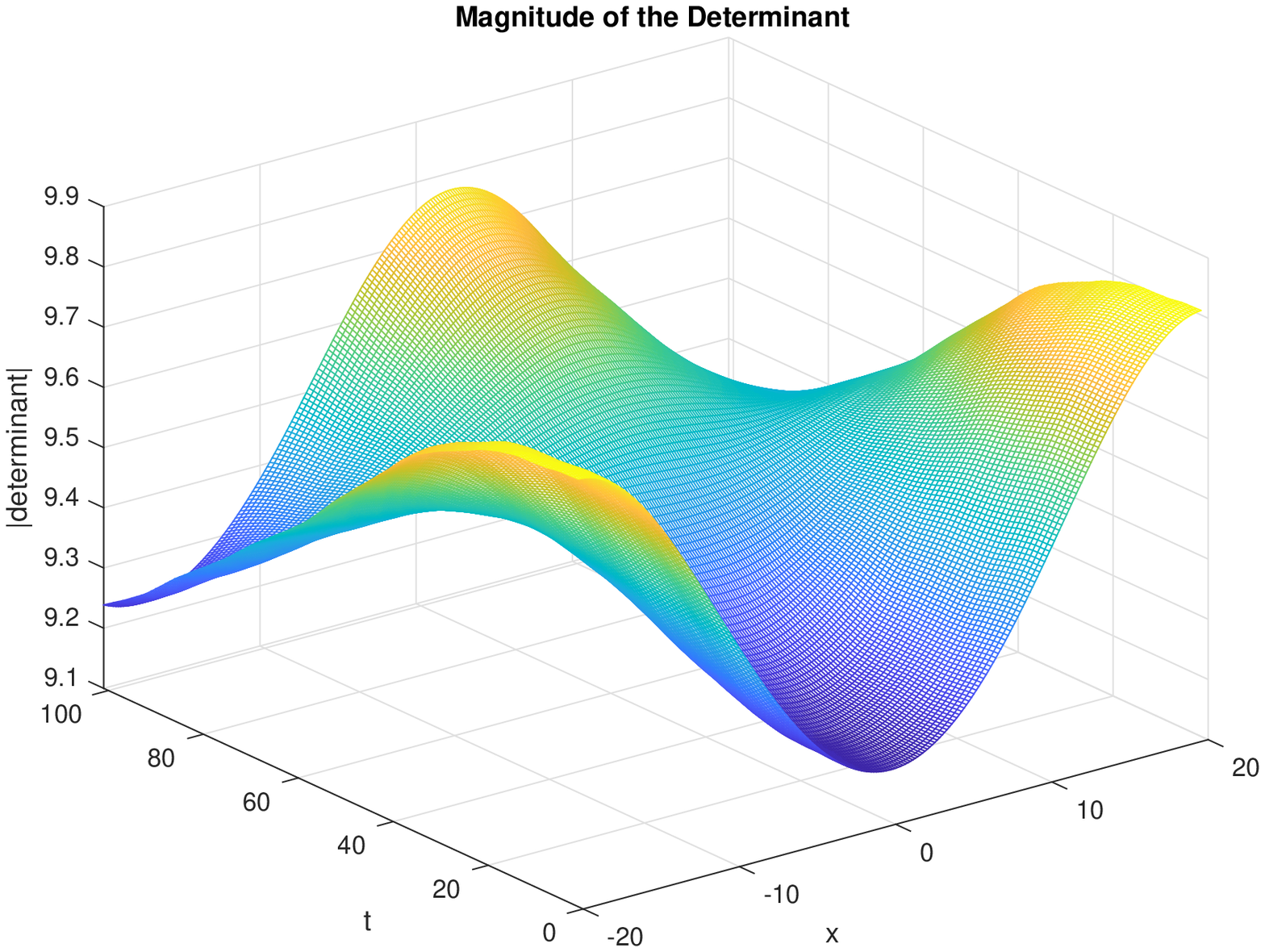}\\
  \end{center}
\caption{We plot the solution to the commutative version of the
fourth order quintic nonlinear Sch\"odinger equation from Theorem~\ref{thm:main}.
We chose the parameter values $\mu_2=-\mathrm{i}$, $\mu_3=1$ and $\mu_4=\mathrm{i}$.
The top panels show the real and imaginary parts computed using a direct integration 
approach, while the middle panels show the corresponding real and imaginary parts 
computed using the Grassmann--P\"oppe method. 
The bottom left panel shows the magnitude of the difference between the 
two computed solutions. The bottom right panel shows the evolution
of the magnitude of the Fredholm determinant of $Q=Q(x,t)$.}
\label{fig:qNLS}
\end{figure}

\begin{remark}[Periodic boundary conditions]
Our main result in Theorem~\ref{thm:main} concerned
the fourth order quintic nonlinear Schr\"odinger equation
on the real line. The numerical simulations above are based
on Fourier spectral approximations on the domain $[-L/2,L/2]$
with periodic boundary conditions. Indeed in the first
step of the Grassmann--P\"oppe method above we generate approximate
solutions to the linear ``base'' partial differential equation
by taking the inverse fast Fourier transform of the exact spectral representation
for the solution. Our numerical simulations shown in Figure~\ref{fig:qNLS}
demonstrate the Grassmann--P\"oppe method appears to work perfectly well 
in the periodic context. However this does require further investigation,
both analytically and numerical analytically. 
\end{remark}

\begin{remark}[Korteweg de Vries and nonlinear Schr\"odinger]
This Grassmann--P\"oppe method was used to generate approximate
solutions to the standard Korteweg de Vries and nonlinear Schr\"odinger equations in
Doikou \textit{et al.\/} \cite{DMSW20}. 
\end{remark}

\begin{remark}[Initial data]\label{rmk:initialdata}
In principle we could generate the numerical solutions
using the Grassmann--P\"oppe approach from given initial
data $g_0=g_0(x)$ for the kernel function $g=g(0,0;x,0)$.
The initial data $p_0$ for the linearised partial differential system
for $p=p(x,t)$ can be computed from $g_0$ via `scattering',
as suggested for example by McKean~\cite{McKean}. 
\end{remark}  

\begin{remark}[Grassmann--P\"oppe single time evaluation]
We emphasise the following efficiency property of the Grassmann--P\"oppe method.
Given the initial data $p_0$, we compute its fast Fourier transform $\hat p_0=\hat p_0(k)$,
for a finite set of wavenumbers $k$. We then advance the individual $k$-modes of $\hat p_0$
to any given time $t>0$ via the Fourier flow map
$\exp\bigl(\Delta t(\mu_2(2\pi\mathrm{i}k)^2+\mu_3(2\pi\mathrm{i}k)^3+\mu_4(2\pi\mathrm{i}k)^4)\bigr)$,
corresponding to the linear partial differential equation prescribing $p=p(x,t)$.
We can thus generate the Fourier coefficients $\hat p=\hat p(k,t)$ at any
given time in one single step. We generate an approximation for $p=p(x,t)$
by then computing the inverse fast Fourier transform. And then finally
we can generate $g=g(y,z;x,t)$ by solving the linear Fredholm equation
at that time $t>0$, as described above. In contrast, the direct numerical
simulation method requires computing the solution over successive small time steps
to evaluate the solution at any time $t>0$.
\end{remark}

\section{Discussion}\label{sec:discussion}
The advantages of the method we present to establish integrability for the
generalised non-commutative fourth order quintic nonlinear Schr\"odinger equation, based on P\"oppe's Hankel
operator approach are as follows. First, the method is abstract. Once the
Fredholm equation $P=G(\id+Q)$ is established the computation proceeds
entirely at the operator level. The key initial ingredients are that the
scattering data $P$ is a Hankel operator and depends on the parameters $x$ and $t$,
and satisfies an evolution in $t$ linear equation involving a derivation operation with respect to $x$.
And then the auxiliary data $Q$ is assigned appropriately in terms of $P$.
Second, with this in hand, we can proceed in the
operator algebra, to which a derivation operation can be applied, once we
endow that algebra with a kernel product rule associated with the Hankel components.
The procedure to establish a closed-form nonlinear
kernel equation is then direct and elementary, only requiring basic calculus.

The observant reader will have noticed in the proof of Theorem~\ref{thm:main}
that, once we computed $\pa_tG-d(\pa)G$ in Step~1, the remaining Steps 2--9
in the proof were a collating exercise, once we applied the kernel
bracket operator at the very beginning of Step~2.
Indeed retrospectively we note the following. Step~2 dealt with the terms
with factor $\mu_2$, i.e.\/ those associated with the second order part of $d(\pa)$.
The key identity in Step~2 helping to establish a closed nonlinear form is identity (i) 
in Lemma~\ref{lemma:keyidentities} for $\pa[PU\tP]$. Step~3 dealt with the terms
with factor $\mu_3$, i.e.\/ those associated with the third order part of $d(\pa)$.
The key identity in Step~3 that established a closed nonlinear form is identity (ii) 
in Lemma~\ref{lemma:keyidentities} for $\pa[PU(\pa\tP)]$, in addition to identity (i).
Then in Steps~4--8, which dealt with the terms with factor $\mu_4$, the key identity
was (iii) in Lemma~\ref{lemma:keyidentities} for $\pa[PU(\pa^2\tP)]$, in addition to the previous two.
Indeed the main work in establishing our main result in Theorem~\ref{thm:main}
was the proof of identities (i)--(iii) in Lemma~\ref{lemma:keyidentities}.
It is likely the proof of the Lemma can be simplified further. 
This suggests a key identity for the quintic order case, i.e.\/ when
the order of $d(\pa)$ is five, will involve an analogous expression 
for $\pa[PU(\pa^3\tP)]$.
This is the last case presented in Nijhoff \textit{et al.\/} \cite{NQLC}.
An explicit closed-form expression for all orders is obviously of interest.
This could perhaps be achieved via a non-commutative generalisation of the recursion relation
for the nonlinear Schr\"odinger hierarchy,
see for example P\"oppe~\cite{P84} or Matveev and Smirnov~\cite{MatveevSmirnov},
and/or via an algebraic combinatorial approach using algebraic structures
analogous to those in Malham and Wiese~\cite{MW},
Ebrahimi--Fard \textit{et al.}~\cite{E-FLMM-KW}
or Ebrahimi-Fard \text{et al.}~\cite{E-FMPW}.

Some final observations are as follows. In Section~\ref{sec:preliminaries}
we discussed how a solution to the linear Fredholm equation for $G$ exists
provided the determinant $\det(\id+Q)\neq0$. This is guaranteed locally
in time under the conditions stated therein. Recall $Q$ is prescribed directly
and solely in terms of $P$ and $P^\dag$. Hence the evolution of $P$ and thus
the determinant $\det(\id+Q)$ determines the existence of $G$. If the
determinant becomes zero at some time $t$ then the solution $G$ may
become singular. More specifically, depending on the route to singularity,
certain eigenvalues of $G$ will become singular; see Beck and Malham~\cite{BM}.
However, such singular behaviour in the context of Grassmannian flows simply
indicates a poor choice of representative coordinate patch. By changing
to another suitable patch, which is always possible, the solution can
be continued. A careful analysis of this scenario is required. 
Additionally a careful numerical analysis of the Grassmann--P\"oppe
method we presented in Section~\ref{sec:numericalsimulations} is
also required.


\section*{Acknowledgement}
SJAM would like to thank Anastasia Doikou and Ioannis Stylianidis for stimulating discussions,
as well as the anonymous referees for their very constructive comments and suggestions that helped
improve the original manuscript.

\end{document}